\documentclass[reqno,11pt]{amsart}
\usepackage[dvipsnames]{xcolor}
\usepackage{geometry,enumitem,comment,bbold}

\usepackage[utf8]{inputenc}
\usepackage{amsmath}
\usepackage{graphicx}
\DeclareGraphicsExtensions{.pdf,.png,.jpg}
\usepackage{setspace}
\usepackage[round,comma]{natbib}         
\usepackage{amsthm}
\usepackage{amsfonts}
\usepackage[normalem]{ulem}  
\usepackage{amssymb}
\usepackage{dsfont}

\usepackage{booktabs}
\usepackage{mathrsfs}
\usepackage{bm} 
\usepackage{mathtools}   
\usepackage{subcaption}

\newtheorem{theorem}{Theorem}[section]
\newtheorem{corollary}[theorem]{Corollary}      %
\newtheorem{lemma}[theorem]{Lemma}              %
\newtheorem{proposition}[theorem]{Proposition}  %

\theoremstyle{definition}
\newtheorem{example}[theorem]{Example} %
\newtheorem{definition}[theorem]{Definition} %
\newtheorem{remark}[theorem]{Remark}%

\numberwithin{equation}{section}

\newcommand{\M}{\mathcal M}
\newcommand{\La}{\Lambda}
\newcommand{\R}{\mathbb{R}} 
\newcommand{\Q}{\mathbb{Q}}

\newcommand{\E}{\mathbb{E}}
\newcommand{\tn}{\textnormal}

\newcommand{\ind}{\mathbf{1}}

\renewcommand{\P}{\mathbb{P}}

\newcommand{\N}{\mathbb{N}}

\newcommand{\eps}{\varepsilon}
\newcommand{\supp}{\textnormal{supp}}

\newcommand{\CF}{\mathcal F}

\newcommand{\CL}{\mathcal L}
\renewcommand{\Q}{{\mathbb Q}}

\newcommand{\w}{\widehat}

\renewcommand{\and}{\quad\text{and}\quad}
\newcommand{\Scal}{\mathcal S}
\newcommand{\G}{\mathcal G}

\newcommand{\quiff}{\quad\iff\quad}
\newcommand{\TV}{\mathrm{TV}}
\newcommand{\gtge}{%
  \mathrel{\ooalign{%
    \hfil$\raise0.7ex\hbox{$\scriptstyle>$}$\hfil\cr
    \hfil$\raise-0.7ex\hbox{$\scriptstyle\geq$}$\hfil\cr
  }}%
}

\DeclareMathOperator{\esssup}{esssup}

\usepackage[colorlinks,urlcolor=red,citecolor=blue,linkcolor=magenta]{hyperref}

\begin{document}

\setstretch{1.3}
\title{Lambda-quantiles under the microscope}

\author[F.~Bellini]{Fabio Bellini}
		\address{Department of Statistics and Quantitative Methods, University of Milano-Bicocca, Italy.}
		\email{fabio.bellini@unimib.it}

    \author[F.-B.~Liebrich]{Felix-Benedikt Liebrich}
		\address{Amsterdam School of Economics, University of Amsterdam, Netherlands.}
		\email{f.b.liebrich@uva.nl}

\begin{abstract}
We study Lambda-quantiles, a generalisation of classical quantiles in which the constant
probability level $\lambda\in[0,1]$ is replaced by a functional parameter
$\Lambda\colon\R\to[0,1]$. We work with a general, non-monotone $\Lambda$, which arises
naturally when the class of Lambda-quantiles is required to be closed under
inf-aggregation or under mixtures. We first characterise finiteness, constancy, and what we
call the attainment property, and we show that $\Lambda$ is identifiable, under mild
assumptions, from the values of its Lambda-quantiles on a two-parameter family of
distributions. We then substantially refine known results on weak upper and lower
semicontinuity and on mixture convexity of the level sets, obtaining almost complete characterisations, which also clarify the domains on which Lambda-quantiles admit a consistent scoring function (are elicitable). For $\Lambda$
of bounded variation we prove a mixture representation: any such Lambda-quantile is a
Lambda-quantile with an increasing parameter, evaluated at a mixture of the original
distribution with a fixed reference distribution at a fixed weight. Finally, we introduce the
ordinal covariance group of a risk measure and show that, for a Lambda-quantile, it coincides
with the invariance group of $\Lambda$ and with a group of measure-preserving transformations
of the signed measure associated with $\Lambda$.

\noindent {\em JEL Classification:} C02 $\cdotp$ D81 $\cdotp$ G32

\noindent {\em Mathematics Subject Classification (MSC2020):} 91G70 $\cdotp$ 60E05
$\cdotp$ 26A45 $\cdotp$ 62P05

\noindent{\em Keywords:} Lambda-quantiles $\cdotp$ weak continuity $\cdotp$ CxLS 
$\cdotp$ elicitability
$\cdotp$  ordinal covariance 
\end{abstract}

\maketitle

\thispagestyle{empty}

\section{Introduction}
Lambda-quantiles, 
originally introduced in the financial literature by \cite{FMP2014}, generalise the classical notion of quantile of a probability distribution by replacing the fixed probability level $\lambda \in [0,1]$ with a functional parameter $\La \colon \R \to [0,1]$. 
More explicitly, the left and right $\La$-quantile of a cumulative distribution function $F\colon\R\to[0,1]$ are respectively defined by
\[
Q_\La^-(F) = \inf \{ x \in \R \mid F(x) \geq \La(x)\}, \quad Q_\La^+(F) = \inf \{ x \in \R \mid F(x) > \La(x)\}.
\]

The simplest case is the one in which $\La$ attains two values, as in the following example:
$$
\La(x) = 
\begin{cases}
\lambda_1 \text { if } x <z,\\
\lambda_2 \text { if } x \ge z, 
\end{cases}
$$
for some $z \in \R$ and $\lambda_1, \lambda_2 \in [0,1]$. In this case, the definition above yields for the left quantile that 
\[Q_\La^-(F) = 
\begin{cases}
Q_{\lambda_1}^-(F) &\text { if } Q_{\lambda_1}^-(F) <  z,\\
z &\text { if } Q_{\lambda_2}^- (F) \leq z \leq Q_{\lambda_1}^- (F),\\ 
Q_{\lambda_2}^-(F) &\text { if } Q_{\lambda_1}^-(F) \geq z\text { and } Q_{\lambda_2}^-(F) > z,\end{cases}\]
where $Q_\lambda^-$ is the usual left quantile.
This example illustrates in the simplest form the key idea underlying the definition of $\La$-quantiles, that is to allow multiple probability levels. 
The three cases correspond to the probability level
jumping from $\lambda_1$ to $\lambda_2$ at $z$, and the $\La$-quantile follows
the classical $\lambda_1$- or $\lambda_2$-quantile away from $z$, and is equal to
$z$ in the intermediate regime. In full generality, the probability level $\La(x)$ is a function of the same variable $x \in \R$
at which $F$ is evaluated. 

Lambda-quantiles are attracting growing attention in the financial and statistical literature for several reasons, among which we believe the following are worth stressing. First, like the usual quantiles and in contrast to common financial risk measures such as the Expected Shortfall, they are defined also for random variables with infinite mean, enabling their application to very heavy-tailed phenomena. 
Second, they are often the minimisers of an expected loss function that generalises that of the usual quantiles, so they are amenable to statistical procedures similar to quantile regression. Third, 
they often have qualitative robustness properties similar to those of the usual quantiles, and  
the statistical properties of empirical Lambda-quantiles can be derived along the lines of those of the usual quantiles.  
Fourth, they can be backtested with methodologies that extend the usual approach based on counting the number of violations of Value at Risk. 
Finally, they enjoy a rich mathematical theory, with different axiomatic foundations, mainly in the case of a decreasing $\La$. 

Early papers such as \cite{FMP2014}, \cite{BPR2017},
\cite{CP2018}, \cite{HMP2018} and \cite{Ince}, focused on the financial
applications of $\La$-quantiles as an alternative to Value at Risk, to which they
refer under the name of Lambda-VaR. The motivation was to give the risk manager
greater flexibility by allowing the confidence level to vary across loss
scenarios, rather than fixing a single probability level, while remaining inside a VaR-like framework. Within this line of research, several practical
issues were addressed, including the choice of the functional parameter $\La$,
the discovery of backtesting methodologies for Lambda-VaR that extend the usual violation-counting
approach, and the analysis of portfolio risk contributions.

More recent works such as \cite{Liu}, \cite{LTW24}, \cite{HL2026}, \cite{LS2025}, \cite{Boonen}, 
\cite{XH2025}, \cite {PW2026}
deal with financial and actuarial applications such as risk-sharing, robust versions under distributional uncertainty, optimal insurance design, numerical methods for Lambda-quantiles and portfolio optimization.  

Another driver behind the increasing interest in Lambda-quantiles is the discovery that in the case of a decreasing $\La$ the properties of $\La$-quantiles are very tractable and constitute a nice generalisation of those of the usual quantiles. Further, in this case $\La$-quantiles
have a nice axiomatic foundation based on peculiar properties of the usual quantiles. More specifically, \cite{BP} provided an axiomatisation based  on the \emph{locality} property. Given the set $\M$ of all univariate cumulative distribution functions and a functional $\rho \colon \M \to \R$, this property requires all $F,G\in\M$ to satisfy
$$
F=G \text{ on } (a,b) \text{ with } \rho(F) \in (a,b)\quad\implies\quad\rho(G)=\rho(F). 
$$ 
In the same context, 
\cite{CMWW} provide an axiomatisation based on the properties of \emph{max-stability} and \emph{min-stability} which require that 
$$
\rho(F \vee G) = \rho(F) \vee \rho(G), \quad
\rho(F \wedge G) = \rho(F) \wedge \rho(G).
$$
Again in the case of a decreasing $\La$, \cite{HWWX} gave inf-sup and sup-inf representations of $\La$-quantiles---defined on random variables via their distribution functions under a reference probability measure---in terms of the usual quantiles:
\[
Q_\La^-(X) = \inf_{x \in \R} \{Q_{\La(x)}^-(X) \vee x \} =
\sup_{x \in \R} \{Q_{\La(x)}^-(X) \wedge x\}.\]
These result in one of the approaches for introducing $\La$-Expected Shortfall in \cite{LambdaES}. 

Crucially, the aforementioned contributions differ in the properties they require the function $\La$ defining the associated quantile to satisfy.
While assumptions of monotonicity are the most common, \cite{Ince}, \cite{LTW24}, \cite{Liu}, and \cite{HL2026} consider non-monotone $\La$. 
The present paper therefore provides a first systematic study of $\La$-quantiles beyond the
monotone case. Our guiding principle is to avoid as much as possible \emph{ad hoc} restrictions on $\La$, and
to aim to derive results in full generality.

Our main findings can be categorised as follows. First, we refine results already known in the literature concerning weak semicontinuity properties of $\La$-quantiles, that are very important to assess their qualitative robustness and the properties of their empirical estimators. In the literature it is known that in the case of a decreasing $\La$ these properties are similar to those of the usual quantiles; in this paper we refine \cite{BP}, \cite{FMP2014}, and  \cite{BPR2017} by providing almost ``if and only if" conditions for weak upper and lower semicontinuity. 

Second, we discuss the validity of the CxLS property, that in general stands for the convexity of the level sets of a functional with respect to mixtures, that is a well-known necessary condition for elicitability, i.e., the property of being the minimiser of a suitable expected loss. This issue has been discussed in \cite{BPR2017}, here we improve some of the results obtained there. 
Since the CxLS property is necessary for elicitability, these characterisations
identify the cases in which no scoring function can exist.  
In the opposite
direction, we show that on the domain of distributions crossing $\La$
exactly once, $\La$-quantiles are elicitable for every measurable $\La$,
extending \cite[Theorem 8]{BPR2017}.
Third, we establish a general mixture representation of $\La$-quantiles. The prototypical
result, to the best of our knowledge first communicated in \cite{W2024}, states that for a
decreasing $\La$ the $\La$-quantile of $F$ coincides with a classical quantile of a mixture
$\lambda F+(1-\lambda)G$ with a fixed reference distribution $G$ at a fixed weight
$\lambda$; as \cite{W2024} pointed out, this reveals a remarkable connection with the theorem of \cite{Moulin} characterising
anonymous, efficient and strategy-proof voting schemes as generalised medians. We prove a
substantial generalisation, requiring only that $\La$ has bounded variation and satisfies the
attainment condition: in this case $Q_\La^{\pm}$ can be represented through an \emph{increasing}
parameter evaluated at a fixed mixture, thus reducing the complexity of $\La$ from bounded
variation to monotone. We complement this with a converse closure result, showing that the
corresponding class of $\La$-quantiles is closed under mixing with a fixed distribution at a
fixed weight; as for aggregation, this closure underlines the need of general, non-monotone $\La$. 

Fourth, we develop a theory of the ordinal covariance of $\La$-quantiles. Ordinal
covariance---the requirement that a functional commute with every strictly increasing,
continuous bijection of $\R$---is one of the defining features of classical quantiles and
underlies their axiomatisations in \cite{C09} and \cite{FLW23}. Rather than treating it as an
all-or-nothing property, we regard it as admitting a continuum of degrees, which we measure
through the \emph{ordinal covariance group} of a functional. Our main result identifies this
group, in the case of a $\La$-quantile, with the invariance group of $\La$; and, whenever $\La$
has bounded variation and is right-continuous, with the pointwise stabiliser of the support of
the signed measure associated with $\La$, using a measure-preservation argument based on the
Poincar\'e recurrence theorem. The degree of ordinal covariance is thus encoded in a single
closed set, ranging from the full group of the classical quantiles, corresponding to a constant
$\La$, down to the trivial group of a strictly monotone $\La$.

The paper is organised as follows. Section~\ref{sec:prelim} gives preliminary definitions and discusses first properties such as finiteness, constancy and what we call the attainment property, i.e., the possibility of replacing the infimum by a minimum in the definition of left Lambda-quantiles. 
Section~\ref{sec:identify} addresses the problem of reconstructing $\La$ from the values of Lambda-quantiles on some suitable set of simple distributions. It turns out that under mild assumptions this can always be done with the two-parameter family of shifted Bernoulli distributions. 
Section~\ref{sec:aggr} studies infimum and supremum of $\La$-quantiles,
showing that the aggregation rules of the classical quantiles extend to
a general $\La$, and that the closure of the class of Lambda-quantiles under these operations naturally leads to
the consideration of non-monotone $\La$.
Section~\ref{sec:weak} provides complete characterisations of the weak lower and upper
semicontinuity and of the weak continuity of both $\La$-quantiles, sharpening earlier
sufficient conditions into equivalences.
Section~\ref{sec:CxLS} characterises the convex level set property for both the left and the
right $\La$-quantile, refining the sufficient conditions of \cite{BPR2017}.
Section~\ref{sec:MR} establishes the general mixture representation, reducing any bounded-variation
$\La$ to an increasing one, and proves the corresponding closure
of the class under mixing.
Section~\ref{sec:OC} develops the theory of ordinal covariance, identifying the ordinal
covariance group of a $\La$-quantile with the invariance group of $\La$ and, in the
bounded-variation case, with the pointwise stabiliser of the support of the associated signed
measure. 
Finally, Section~\ref{sec:concl} concludes and discusses directions for further research, and Table \ref{tab:summary} provides a synopsis of our main mathematical results.

\section{Preliminaries and first properties}\label{sec:prelim}

We denote by $\M$ the set of cumulative distribution functions (CDFs) on $\R$ and by $\M_c$ its subset of compactly supported distributions.
Given a function $f\colon 
\R\to\R$, we write
$f(y+)=\lim_{x\downarrow y}f(x)$ and $f(z-)=\lim_{x\uparrow z}f(x)$
for $y\in[-\infty,\infty)$ and $z\in(-\infty,\infty]$, whenever these limits exist.  
We say that $f$ is \emph{right-regular} (respectively, \emph{left-regular}) if it admits right (respectively, left) limits everywhere. We say that $f$ is \emph{right} (respectively, \emph{left}) \emph{lower semicontinuous} if $\liminf_{y \downarrow x} f(y) \geq f(x)$ (respectively, $\liminf_{y \uparrow x} f(y) \geq f(x)$). 
The terms increasing and decreasing are used in the weak sense.
Otherwise, we speak of strictly increasing or decreasing functions. As usual, we set 
$ \inf \varnothing = \infty$. 

Throughout the manuscript, $\La$ denotes a function mapping $\R$ to $[0,1]$, which is used to define $\La$-quantiles as follows. 

\begin{definition}[Lambda-quantiles]
\label{def:La-Q}
Let $\La \colon \R \to [0,1]$ and $F \in \M$. Then 
\begin{equation*}
Q_\La^-(F) :=\inf \left\{ x\in \R\mid F(x)\geq \La
(x)\right\}\text { and }
Q_\La^+(F):=\inf \left\{ x\in \R\mid F(x)>\La (x)\right\}
\end{equation*}  
are respectively called the left and the right $\La$-quantiles of $F$.
\end{definition}

Lambda-quantiles generalise classical quantiles that arise when $\La\equiv\alpha \in [0,1]$.   
By definition, clearly
$
Q_\La^-\le Q_\La^+.
$
Notice that if $\La(a)=0$ for some $a\in\R$, it holds that $Q_\La^-(F) \leq a$ for each $F \in \M$, so the values of $\La$ for $x > a$ are immaterial for the determination of $Q_\La^-$.
This motivates the following.

\begin{definition}[Standard representation]

Let $\La \colon \R \to [0,1]$. Let
\begin{equation}
\label{eq:N-zero}
N_0 := \inf \{ x \in \R \mid \La(x)=0 \},
\end{equation}
with $\inf \varnothing = \infty$. 
We say that $\La$ is {\em in standard representation} if $\La(x)=0$ for all $x > N_0$. 
\end{definition}

The usual quantiles are always finite when $\alpha \in (0,1)$, and it holds that $Q_0^- \equiv -\infty$, $Q_1^+ \equiv \infty$, and $Q_0^+(F)$ and $Q_1^-(F)$ are the minimum and maximum of the support of $F$, i.e., its essential infimum and supremum. 
Finiteness of Lambda-quantiles is characterised by the following lemma. 

\begin{lemma}[Finiteness]
\label{lem:liminf}
Let $\La\colon\R\to[0,1]$. Then,
\begin{equation}
\liminf_{x \to -\infty}\Lambda(x)>0 \iff Q_\La^->-\infty\iff Q_\La^+>-\infty,\label{eq:equiv1}\end{equation}
and if $\La$ is in standard representation,
\begin{equation}\liminf_{x\to \infty}\La(x)<1 \iff Q_\La^+<\infty\iff Q_\La^-<\infty.\label{eq:equiv2}
\end{equation}
\end{lemma}
\begin{proof}
Regarding \eqref{eq:equiv1}, $Q_\La^->-\infty$ implies $Q_\La^+>-\infty$. 
Next, assume by contradiction the existence of a sequence $x_n \to -\infty$  such that  $\La(x_n) \to 0$. 
Let $F\in\mathcal M$ be any CDF with the property $F(x_n)>\La(x_n)$ for infinitely many $n$. Then 
$Q_\La^+(F)\le \inf_{n\in\N}x_n=-\infty$, a contradiction. 
Finally, if $\liminf_{x \to -\infty}\La(x)>0$, there exist $\delta>0$ and $x\in\R$ such that $\La(y)\ge \delta$ for all $y\le x$.
As $F(-\infty)=0$ for all $F\in\mathcal M$,
this suffices to guarantee that $Q_\La^->-\infty$. 

For \eqref{eq:equiv2}, note that $Q_\La^+<\infty$ implies $Q_\La^-<\infty$. 
If $\liminf_{x\to\infty}\La(x)=1$, there is a sequence $(x_n)$ increasing to $\infty$ such that $\La(x)\ge 1-3^{-n}$ holds for all $x\ge x_n$. 
Let 
$
\nu:=\sum_{n=1}^\infty \delta_{x_n}/2^n.
$
Using that $\La$ is in standard representation, the associated CDF $F_\nu$ satisfies  $F_\nu\equiv 0<\La$ on $(-\infty,x_1)$ and 
$
F_\nu(x)=1-2^{-n}<\La(x)
$
for
$x\in[x_n,x_{n+1})$.
Hence, $Q_\La^-(F_\nu)=\infty$. 
Last, if $\liminf_{x\to\infty}\La(x)<1$,  $Q_\La^+<\infty$ follows because $F(\infty)=1$, $F\in\mathcal M$. 
\end{proof}

If $\La$ is decreasing, condition~\eqref{eq:equiv1} is equivalent to $\La$ not being identically equal to $0$, while \eqref{eq:equiv2} is equivalent to $\La$ not being identically equal to $1$, which were the conditions for finiteness in \cite{BP}. 
In \cite{BPR2017}, the more restrictive assumption that $\La$ is bounded away from both $0$ and $1$ was required, similarly to the original assumptions in~\cite{FMP2014}.

Recall that a functional $\rho \colon \M \to \R$ is said to satisfy the \emph{constancy} property if $\rho(F_c)=c$ for all $c \in \R$, where $F_c:=\ind_{[c,\infty)}$ is the CDF of the Dirac measure at 
$c$. Lambda-quantiles may fail to have this property if $\La$ takes the values $0$ or $1$; the following lemma gives precise equivalent conditions.

\begin{lemma}
 [Constancy]

Let $\La\colon\R\to[0,1]$, $N_0$ be defined as in \eqref{eq:N-zero}, and $c\in\R$. Then:
\begin{enumerate}[label=\tn{(\alph*)}]
    \item It holds that $Q_\La^-(F_c)=c$ if and only if $\La>0 \text{ on }(-\infty,c)$. 
    In particular, $Q_\La^-$ satisfies the constancy property 
if and only if 
    $N_0=\infty$.
    \item It holds that $Q_\La^+(F_c)=c$ if and only if $\inf\{y\ge c\mid \La(y)<1\}=c$.
    In particular, $Q_\La^+$ satisfies the constancy property if and only if the set 
    $\{\La<1\}$ is dense in $\R$.
\end{enumerate}
\end{lemma}
\begin{proof}
    Recall the definition $F_c=\ind_{[c,\infty)}$, $c\in\R$. 
    For statement (a), one has 
    \[\{x\in\R\mid F(x)\ge\La(x)\}=[c,\infty)\cup\{y<c\mid\La(y)=0\},\] 
    whose 
    infimum equals $c$ exactly when $\La(y)>0$ for all $y<c$. 
    Hence, the constancy property holds throughout if and only if $\La>0$ on $\R$, which is tantamount to $N_0=\infty$.

    Regarding statement (b), it holds that $\{y\in\R\mid F(y)>\La(y)\}=\{y\ge c\mid\La(y)<1\}$, 
    whose infimum is $c$ if and only if every interval $[c,c+\eps)$ intersects 
    $\{\La<1\}$. This holds for all $c$ if and only if $\{\La<1\}$ is dense.
\end{proof}

A remarkable property of the usual left quantiles which follows from right-continuity of distribution functions 
is that the infimum arising in Definition \ref{def:La-Q} is actually a minimum, i.e., $F(Q_\alpha^-(F))\ge \alpha$.    
We call this the \emph{attainment property}, and give an equivalent condition in the lemma below. 

\begin{lemma}[The attainment property]
\label{lem:attain}
    Let $\La \colon \R \to [0,1]$ be in standard representation. The following are equivalent: 
    \begin{enumerate}[label=\tn{(\alph*)}]
        \item $\La$ is right lower semicontinuous.
        \item For all $F\in\mathcal M$ with $Q_\La^-(F)\in\R$,  
       $
            F\big(Q_\La^-(F)\big)\ge \La\big(Q_\La^-(F)\big).
  $
    \end{enumerate}
\end{lemma}
\begin{proof}
    (a) implies (b): Suppose $F\in\mathcal M$ is such that $q:=Q_\La^-(F)\in\R$.
    The statement is clear if $q\in\{F\ge \La\}$.
    Else, there exists a sequence $(z_n)\subseteq\{F\ge\La\}$ strictly decreasing to $q$. Hence,
$
F(q)=\lim_{n\to\infty}F(z_n)\geq \liminf_{n \to \infty} \La(z_n) \geq \La(q)
$
by assumption (a).

(b) implies (a): By contradiction, suppose we can find $x\in\R$ with $\La(x)>\lambda:=\liminf_{y\downarrow x}\La(y)$. 
In particular, $\La(x)>0$.
     Select a decreasing sequence $(x_n)\subseteq (x,\infty)$ converging to $x$ such that $\La(x_n)$ converges to $\lambda$. 
    Define a CDF $F$ by setting 
$F|_{(-\infty,x)}\equiv 0,\, F(x)=\lambda,\, F|_{[x_{n+1},x_n)}=\sup_{k\ge n+1}\La(x_k)$, $F|_{[x_1,\infty)}\equiv 1$.
By construction, $F(x_n)\ge \La(x_n)$ for all $n\in\N$, and $F(y)=0<\La(y)$ for all $y<x$ because $\La$ is in standard representation. This means that $Q_\La^-(F)=x$, but $F(x)=\lambda<\La(x)$. 
\end{proof}

\section{Identification and reconstruction of \texorpdfstring{$\La$}{Lambda}}\label{sec:identify}

A natural question is whether the function $\La$ is uniquely determined by the values of Lambda-quantiles, and what is an efficient procedure for reconstructing $\La$ from Lambda-quantiles of distributions. 
As discussed in Section \ref{sec:prelim}, all the different $\La$ having the same standard representation produce the same left $\La$-quantiles, so to prevent this obvious lack of identifiability we assume throughout this section that $\La$ is in standard representation, whenever $Q_\La^-$ is involved. 
Another source of lack of identifiability is Proposition 2.6 d) in \cite{BP}, showing that in the case of a decreasing $\La$ all functions $\Lambda'$ satisfying 
$
\La' = \La$ 
on the continuity points of $\La
$
satisfy $Q_{\La'}^{\pm} = Q_\La^{\pm}$. 

Let us consider the following two-parameter family of shifted Bernoulli distributions. 
\begin{definition}[Shifted Bernoulli]
For $x\in\R$ and $p \in [0,1]$ we set
\begin{equation}
\label{eq:Bernoulli}
B_{x,p}:=(1-p)\ind_{[x,x+1)}+\ind_{[x+1,\infty)}
\end{equation}
and 
$
\mathcal M_B:=\{B_{x,p}\mid x\in\R,\,p\in[0,1]\}.
$
\end{definition}

The following lemma shows that $N_0$ is determined by the values of $Q_\La^-$ on one-point distributions. 

\begin{lemma}[Finding $N_0$]
\label{lem:find N0}
Let $N_0$ be defined as in \eqref{eq:N-zero}. If $\La$ is in standard representation, then
\[N_0=\inf\{x\in\R\mid Q_\La^-(B_{x,0})\neq x\}.\]
\end{lemma}

\begin{proof}
Let $S:=\{x\in\R\mid Q_\La^-(B_{x,0})\neq x\}$, and recall that $\La>0$ on $(-\infty,N_0)$ while by standard representation $\La\equiv 0$ on $(N_0,\infty)$.
If $\La(x)>0$, then $x\le N_0$ and hence $\{\La\le B_{x,0}\}=[x,\infty)$, so that $Q_\La^-(B_{x,0})=x$. Thus $S\subseteq\{\La=0\}$ and $\inf S\ge N_0$.
Conversely, let $x>N_0$. Then
 $Q_\La^-(B_{x,0})=N_0<x$ and $x\in S$; in particular $Q_\La^-(B_{x,0})=-\infty$ for every $x\in\R$ when $\La\equiv 0$. Therefore $\inf S\le N_0$.
\end{proof}

The following proposition shows that under the attainment condition, the values of $\La(x)$ for $x\le N_0$ are determined by the values of left $\La$-quantiles on the two-parameter set of shifted Bernoulli distributions. 

\begin{proposition}[Reconstruction of $\La$ from $Q_\La^-$]
\label{prop:geq}
    If $\La$ is in standard representation, the following are equivalent 
    \begin{enumerate}[label=\tn{(\alph*)}]
        \item $\Lambda$ is right lower semicontinuous 
        \item For all $x\in\R$ with $x\le N_0$, it holds that
        $
      \La(x)=\inf\{1-p\mid Q_\La^-(B_{x,p})=x\}
    $
        \item For all $x\in\R$ with $x\le N_0$, it holds that
      $\La(x)=\inf\{F(x)\mid Q_\La^-(F)=x\}$.    
    \end{enumerate}
\end{proposition}
\begin{proof}
    (a) implies (b): Let $x\in\R$ be such that we can choose $p\in[0,1]$ with $Q_\La^-(B_{x,p})=x$. Lemma~\ref{lem:attain}
    then implies $1-p=B_{x,p}(x)\ge \La(x)$.
    If $\La(x)>0$, 
    the infimum in (b) is attained by $B_{x,1-\La(x)}=\La(x)\ind_{[x,x+1)}+\ind_{[x+1,\infty)}$. 
If $\La(x)=0$, $\La$ being in standard representation implies that $x=N_0$. 
    In this case $Q_\La^-(B_{N_0,p})=N_0$ for every $p\in[0,1]$, so the infimum reduces to 
    $\inf_{p\in[0,1]}B_{N_0,p}(N_0)=\inf_{p\in[0,1]}(1-p)=0=\La(N_0)$, 
    the value being approached by $B_{N_0,p}$ as $p\uparrow 1$.
    
    (c) implies (a): For all $x\in\R$ with $x>N_0$, there is no $F\in\M$ with $Q_\La^-(F)=x$. For all $x\le N_0$ and $F\in\M$ with $Q_\La^-(F)=x$, statement (c) implies that $F(x)\ge \La(x)$. 
    Statement (a) then follows with Lemma~\ref{lem:attain}.
    
    (b) implies (c): As $\mathcal M_B\subseteq\M$ and $B_{x,p}(x)=1-p$, statement (b) gives 
    \[\inf\{F(x)\mid Q_\La^-(F)=x\}\le\inf\{1-p\mid Q_\La^-(B_{x,p})=x\}=\La(x).\]
    For the reverse inequality, let $F\in\mathcal M$ be such that $Q_\La^-(F)=x\le N_0$ and, 
    towards a contradiction, suppose that $F(x)<\La(x)$.
    Select a sequence $y_n\downarrow x$ such that $\La(y_n)\le F(y_n)$, $n\in\N$. In particular, 
    \[\liminf_{n\to\infty}\La(y_n)\le F(x)<\La(x).\]
    Select $p\in(0,1)$ such that $F(x)<1-p<\La(x)$. 
    By construction, $Q_\La^-(B_{x,p})=x$, but this poses a contradiction to (b).  
\end{proof}

In conclusion, under right lower semicontinuity, Lemma~\ref{lem:find N0} together with Proposition \ref{prop:geq} yields the following procedure for reconstructing $\La$ from the values of $\rho:=Q_\La^-$ on shifted Bernoulli distributions:
\begin{enumerate}[label=(\roman*)]
\item First,  determine $N_0$ as 
$
N_0 = \inf \{x\in \R\mid \rho(B_{x,0}) \ne x\}
$.
\item Second, set $\La(x)=0$ for all $x > N_0$.
\item Third, for each $x \le N_0$ determine $\La(x)$ with Proposition \ref{prop:geq}(b). 
\end{enumerate}
As a consequence,  
within the class of right lower semicontinuous $\La$ in standard 
representation, the two-parameter family $\mathcal M_B$ is identifying: the values of the left $\La$-quantile on $\mathcal M_B$ 
determine $\La$.
Outside this class, the procedure may fail to uniquely identify $\La$, as the following example shows. 

\begin{example}

    For $\eps\in[\frac 1 2,1]$, consider 
    \[\La_\eps(x)=\begin{cases}0.9&x>0.4,\\[-0.6ex]
    x+0.5&0<x\le 0.4,\\[-0.6ex]
    \eps&x=0,\\[-0.6ex]
    1&x<0,\end{cases}
    \]
    a function that is right lower semicontinuous if and only if $\eps=\frac 12$.
    Assume $\eps,\delta\in(\frac 1 2,1)$.
    Whenever $Q^-_{\La_\eps}(F)\neq 0$, then $Q^-_{\La_\eps}(F)=Q^-_{\La_\delta}(F)$.    
    If $Q^-_{\La_\eps}(F)=0$, we distinguish two cases. 
    Either $F(0)\ge \eps$, in which case we also have $F(x)\ge x+0.5$ for all $0<x<\eps - \frac 1 2$, so that $Q^-_{\La_\delta}(F)=0$. 
    Else, if $F(0)<\eps$, there exists a decreasing sequence $(x_n)\subseteq(0,0.4)$ such that $F(x_n)\ge x_n+0.5$. Again, $Q^-_{\La_\delta}(F)=0$.
    In sum, the values of $Q^-_{\La_\eps}$ do not contain any information about the value $\eps$. 
    The boundary case $\eps = \frac 1 2$ can be viewed as the right lower semicontinuous version of each $\La_{\eps}$, and Proposition \ref{prop:geq} applies. 
    Nevertheless, the resulting left $\La$-quantile differs in that $Q^-_{\La_{0.5}}\neq Q^-_{\La_\eps}$ for all $\eps\in(\frac 1 2,1]$. 
    This can be verified by plugging in $G=\frac 12\ind_{[0,1)}+\ind_{[1,\infty)}$ and computing 
$Q^-_{\La_{1/2}}(G)=0$ and $Q^-_{\La_\eps}(G)=1$.
\end{example}

It is instructive to examine in more detail the output of the reconstruction procedure in the case in which $\La$ 
is decreasing but possibly not right lower semicontinuous, so identifiability may fail. Let
$
\La^\ast(x):=\lim_{y \downarrow x} \La(y) 
$ 
be the right-continuous version of $\La$. 
As recalled at the beginning of this section, replacing $\La$ by $\La^\ast$
leaves the values of $Q_\La^-$ unchanged. Since Proposition \ref{prop:geq}(b) depends only on the values of $Q_\La^-$, it returns the same function whether it is applied 
to $\La$ or to $\La^\ast$. Applied to the latter, which is in particular right lower semicontinuous, it returns $\La^\ast$. Hence, in the decreasing case, the procedure
always automatically reconstructs the right-continuous version of $\La$.

Finally, Proposition \ref{prop:geq-2} below shows that a similar reconstruction procedure can be done starting from 
right Lambda-quantiles. The statement is slightly different from Proposition~\ref{prop:geq} since the points not 
arising as possible values of $Q_\La^+$ no longer form the half-line $(N_0,\infty)$ as in the case of $Q_\La^-$, but 
are now the complement of the set $\mathcal R$ below, on which $\La$ must be equal to $1$; this 
is why the reconstruction formula is stated by cases rather than restricted to the set
$x\le N_0$.

\begin{proposition}[Reconstruction of $\La$ from $Q_\La^+$]
\label{prop:geq-2}
Let $\La \colon \R \to [0,1]$ and 
$
\mathcal R:=\{x\in\R\mid Q_\La^+(B_{x,0})=x\}.
$
The following are equivalent:
    \begin{enumerate}[label=\tn{(\alph*)}]
        \item $\La$ is right lower semicontinuous.
        \item For all $x\in\R$, 
        \begin{equation*}
      \La(x)=\begin{cases}\inf\{1-p\mid Q_\La^+(B_{x,p})=x\}&~~\text{if }x\in\mathcal R,\\[-0.6ex]
      1&~~\text{otherwise}.\end{cases}
    \end{equation*}
        \item For all $x\in\R$, \begin{equation*}
      \La(x)=\begin{cases}\inf\{F(x)\mid F\in\mathcal M,\,Q_\La^+(F)=x\}&~~\text{if }x\in\mathcal R,\\[-0.6ex]
      1&~~\text{otherwise}.\end{cases}
    \end{equation*}
    \end{enumerate}
\end{proposition}

\begin{proof}
    (a) implies (b): Assume first that $x\in\mathcal R$, i.e., 
$\inf\{1-p\mid Q_\La^+(B_{x,p})=x\}\le 1.$
    If this infimum agrees with 1, for all $p>0$, $Q_\La^+(B_{x,p})\neq x$ must hold. 
    By construction, only $Q_\La^+(B_{x,p})>x$ is possible, which means that 
$\La(x)\ge B_{x,p}(x)=1-p.$
    Hence, $\La(x)=1$ and agrees with the infimum. 
    Else, if the infimum is less than 1, select a sequence $(p_n)\subset(0,1)$ such that $1-p_n$ decreases to said infimum and $Q_\La^+(B_{x,p_n})=x$ for all $n$. 
    For a suitable sequence $(y_k)\subseteq[x,\infty)$ converging to $x$, 
\[\La(x)\le\liminf_{y\downarrow x}\La(y)\le \liminf_{k\to\infty}\La(y_k)\le\liminf_{k\to\infty}B_{x,p_n}(y_k)=1-p_n.\]
    Letting $n\to\infty$ delivers
$\La(x)\le\inf\{1-p\mid Q_\La^+(B_{x,p})=x\}<1.$
    If $1-p$ undercuts said infimum, $Q_\La^+(B_{x,p})>x$ must hold, which means that $\La(x)\ge 1-p$. In total, we obtain 
$\La(x)=\inf\{1-p\mid Q_\La^+(B_{x,p})=x\}$.
    Next, if $x\notin\mathcal R$, then $Q_\La^+(B_{x,0})>x$ must hold. 
    This means that $\La(x)\ge B_{x,0}(x)=1$. 
    
    (b) implies (c): 
    By construction, 
    \[\inf\{1-p\mid Q_\La^+(B_{x,p})=x\}\ge\inf\{F(x)\mid F\in\M,\,Q_\La^+(F)=x\}.\]
    Conversely, suppose $F\in\mathcal M$ satisfies $Q_\La^+(F)=x$; there is a sequence $(y_n)\subseteq[x,\infty)$ converging to $x$ such that $F(y_n)>\La(y_n)$ for all $n\in\N$.
    Hence, for all $n\in\N$, 
$Q_\La^+(B_{x,1-F(y_n)})=x$, which entails
    \[\inf\{1-p\mid Q_\La^+(B_{x,p})=x\}\le\inf_{n\in\N}F(y_n)=F(x).\]

    (c) implies (a): By contradiction, assume that (a) fails, i.e., there exists $x\in\R$ such that 
$
\Lambda(x) > \lambda := \liminf_{y \downarrow x} \Lambda(y).
$
Then $\lambda<\La(x)\le 1$, so there is a sequence $y_n\downarrow x$ with $\La(y_n)<1$; this yields $Q_\La^+(B_{x,0})=x$, that is, $x\in\mathcal R$. 
Choose $p \in (0,1)$ with $\Lambda(x) > 1 - p > \lambda$. 
By construction, $Q_\La^+(B_{x,p}) = x$, while $B_{x,p}(x)=1-p<\Lambda(x)$, contradicting~(c).
\end{proof}

For comparison, for the usual quantiles the parameter $\lambda$ can be uniquely identified by the values on the one-parameter family of Bernoulli random variables $B_{0,p}$, or on any one-parameter family of two-point distributions supported on two given points $a,b \in \R$. The same is true for general distortion risk measures or distortion riskmetrics in the sense of \cite{Distortion1,Distortion2}, where the distortion function can be reconstructed by the values of the functional on the one-parameter family $B_{0,p}$. Lemma \ref{lem:find N0} and Proposition~\ref{prop:geq} together show that as a result of the lack of translation invariance,
$\La$-quantiles are just one degree of freedom more complex; indeed their functional parameter $\La$ can be reconstructed from the values of Lambda-quantiles on shifted Bernoulli distributions. 

\section{Inf- and sup- aggregation}\label{sec:aggr}

In this section, we study aggregation properties of Lambda-quantiles with respect to the operations of infimum and supremum, that often yield again a Lambda-quantile. We begin by recalling the corresponding aggregation properties of the
classical quantiles, that will arise as a special case of our results. Recall that for $F\in\M$, the maps $\alpha\mapsto Q_\alpha^-(F)$ and $\alpha\mapsto Q_\alpha^+(F)$
are increasing on $[0,1]$, the former is left-continuous and the latter
right-continuous, and they satisfy
\[
Q_\alpha^+(F)=\lim_{\beta\downarrow\alpha}Q_\beta^-(F)
\and
Q_\alpha^-(F)=\lim_{\beta\uparrow\alpha}Q_\beta^+(F);
\]
see, e.g., \cite[Appendix~A.3]{FoeSch} or \cite[Proposition~1]{EH13}.
Letting $(\alpha_i)_{i\in\mathcal I}\subseteq[0,1]$, $\underline\alpha:= \inf_i \alpha_i$, $\overline\alpha:= \sup_i \alpha_i$,  the
following aggregation rules hold:
\begin{enumerate}[label=\tn{(\roman*)}]
    \item For an arbitrary index set $\mathcal I$,
$\inf_{i\in\mathcal I}Q_{\alpha_i}^+=Q_{\underline\alpha}^+$
and
$\sup_{i\in\mathcal I}Q_{\alpha_i}^-=Q_{\overline\alpha}^-$.
    \item For a finite $\mathcal I$,
$
    \min_{i\in\mathcal I}Q_{\alpha_i}^-=Q_{\min_{i\in\mathcal I}\alpha_i}^-
$
and
$
    \max_{i\in\mathcal I}Q_{\alpha_i}^-=Q_{\max_{i\in\mathcal I}\alpha_i}^-,
$
    and similarly with $Q^+$ in place of $Q^-$.
\end{enumerate}

Our first result shows that the infimum rule in~(i) and the
minimum rule in~(ii) extend to Lambda-quantiles with a general
functional parameter, and their counterparts are formulae~\eqref{eq:min1} and~\eqref{eq:min2}
below. All the remaining rules do not generally carry over.  

\begin{proposition}[Inf-aggregation]
\label{prop:lattice1}
Let $\La_i\colon\R\to[0,1]$ for $i\in\mathcal I$, where $\mathcal I$ is an arbitrary index set.  
Then
\begin{equation}\label{eq:min1}\inf_{i\in\mathcal I}Q_{\La_i}^+=Q_{\inf_{i\in\mathcal I}\La_i}^+.\end{equation}
If $\mathcal I$ is finite, we also have
\begin{equation}\label{eq:min2}\min_{i\in\mathcal I}Q_{\La_i}^-=Q_{\min_{i\in\mathcal I}\La_i}^-.\end{equation}
\end{proposition}
\begin{proof}
Set $\La:=\inf_{i\in\mathcal I}\La_i$. As $\La\le \La_i$ holds for all $i\in\mathcal I$, we have
$\{F>\La_i\}\subseteq\{F>\La\}$ for all $F\in\mathcal M$, meaning that also $Q_\La^+\le Q_{\La_i}^+$. This shows ``$\ge$'' in equation~\eqref{eq:min1}. 
For the converse inequality, assume first that $Q_\La^+(F)=-\infty$. This means that there is a sequence $(x_n)\subseteq\R$ diverging to $-\infty$ such that $F(x_n)>\La(x_n)$. In particular, for all $n\in\N$ we find $i_n\in\mathcal I$ such that $F(x_n)>\La_{i_n}(x_n)$, meaning that $Q_{\La_{i_n}}^+(F)\le x_n$. In total, 
$
\inf_{i\in\mathcal I}Q_{\La_i}^+(F)\le\inf_{n\in\N}Q_{\La_{i_n}}^+(F)\le\inf_{n\in\N}x_n=-\infty$.
If $\inf_{i\in\mathcal I}Q_{\La_i}^+(F)>-\infty$,
we have for all $x<\inf_{i\in\mathcal I}Q_{\La_i}^+(F)$ and all $i\in\mathcal I$ that 
$
F(x)\le \La_i(x)
$.
Hence, 
$F(x)\le \La(x)$,
meaning that $Q_\La^+(F)\ge x$. 
This suffices to prove ``$\le$'' in \eqref{eq:min1}. 
Now, for \eqref{eq:min2}, the inequality ``$\ge$'' is shown as before. 
For the converse inequality, 
assume first that $Q_\La^-(F)=-\infty$. This means that there is a sequence $(x_n)\subseteq\R$ diverging to $-\infty$ such that $F(x_n)\ge \La(x_n)$. In particular, for all $n\in\N$ we find $i_n\in\mathcal I$ such that $F(x_n)\ge \La_{i_n}(x_n)$. 
As $\mathcal I$ is finite, an index $j\in\mathcal I$ exists such that $F(x_n)\ge \La_j(x_n)$ holds for infinitely many $n$. 
Consequently, $Q_{\La_j}^-(F)=-\infty$. 
If $\min_{i\in\mathcal I}Q_{\La_i}^-(F)>-\infty$, we have for all $x<\min_{i\in\mathcal I}Q_{\La_i}^-(F)$ and all $i\in\mathcal I$ that 
$F(x)<\La_i(x)$.
As $\mathcal I$ is finite, also 
$F(x)< \La(x)$,
meaning that $Q_\La^-(F)\ge x$. 
This suffices to prove ``$\le$'' in \eqref{eq:min2}. 
\end{proof}

In view of formula~\eqref{eq:min2}, one might well hope for the validity of 
\begin{equation}
\max\{Q_{\La_1}^-,Q_{\La_2}^-\}=Q_{\max\{\La_1,\La_2\}}^-.
\end{equation}

The next example shows that max-aggregation does not hold in general.  

\begin{example}
\label{ex:max}
    Suppose $F$ is the CDF of a uniform distribution over $[0,1]$ and let $(x_n)\subseteq(\frac 1 4,\frac 34)$ be a strictly decreasing sequence with limit $\frac 1 4$. Additionally, set $x_0=\frac 34$. 
    Now, define 
    \[\La_1=\tfrac 12\ind_{(-\infty,\frac 1 4]}+\sum_{i=0}^\infty x_{2i}\ind_{(x_{2i+2},x_{2i}]}+\tfrac 3 4\ind_{(\frac 34,\infty)},\]
    and
    \[\La_2=\tfrac 12\ind_{(-\infty,\frac 1 4]}+\sum_{i=1}^\infty x_{2i-1}\ind_{(x_{2i+1},x_{2i-1}]}+\tfrac 3 4\ind_{(x_1,\infty)}.\]
    The sequence $(x_{2i})$ is contained in $\{F\ge \La_1\}$ and the sequence $(x_{2i-1})$ is contained in $\{F\ge \La_2\}$, resulting in 
$Q_{\La_1}^-(F)=Q_{\La_2}^-(F)=\frac 14$.
    Nevertheless, 
$Q_{\max\{\La_1,\La_2\}}^-(F)=\frac 3 4.
$ 
\end{example}
Another important side of Proposition~\ref{prop:lattice1} 

is that it supports our decision to study $\La$-quantiles for general functions $\La$.
Indeed, pointwise minima over a finite family of monotone functions need not be monotone. 
Hence, if decreasing and increasing $\Lambda$ are viewed as natural, incorporating non-monotone $\Lambda$ becomes necessary if the class of $\La$-quantiles is to be closed under minimum and maximum aggregation, paralleling the case of classical quantiles.

\begin{remark}
It is interesting to compare the aggregation studied in Proposition~\ref{prop:lattice1} 

with the results of \cite{HL2026}
and \cite{LS2025}, that aggregate over probabilistic models keeping $\La$ fixed.  
In \cite{HL2026} the objects of interest are
the worst and best $\La$-quantiles over a set of distributions $\mathcal N \subseteq \M$, and their main result shows
that the extremal values are again $\La$-quantiles, evaluated at the
extremal distributions $F_{\mathcal N}^-=\inf_{F\in\mathcal N}F$ and
$F_{\mathcal N}^+=\sup_{F\in\mathcal N}F$. 
A second, complementary form of aggregation over the probabilistic models is
carried out in \cite{LS2025}, where the ambiguity resides in a family of
probability measures rather than in the induced distributions. There the
worst-case $\La$-quantile over such a family is shown, for monotone $\La$, to be
itself a $\La$-quantile computed with respect to a capacity, that is a
non-additive set function. 
\end{remark}

\section{Weak continuity and semicontinuity}\label{sec:weak}

Recall that a sequence of CDFs $(F_n)\subseteq \M$ converges weakly to $F \in \M$ if $F_n(x) \to F(x)$ as $n\to\infty$ at all continuity points of $F$. A functional $\rho \colon \M \to [-\infty,\infty]$ is:
\begin{enumerate}[label=(\alph*)]
    \item weakly upper semicontinuous at $F \in \M$ if $\limsup_{n \to \infty} \rho(F_n) \leq \rho(F)$ whenever $F_n$ converges to $F$ weakly
    \item weakly lower semicontinuous at $F\in \M$ if $-\rho$ is weakly upper semicontinuous at $F$
    \item weakly continuous at $F\in\mathcal M$ if it is both weakly upper and lower semicontinuous at $F$. 
\end{enumerate}
Likewise, $\rho$ is weakly upper or lower semicontinuous if that property holds at all $F\in\M$. 

The weak semicontinuity properties of classical quantiles are well known: the left quantile is weakly lower semicontinuous everywhere; the right quantile is weakly upper semicontinuous everywhere; and both are weakly continuous at those $F \in \M$ where they coincide. 
For decreasing $\La$, \cite{BP} showed that $\La$-quantiles share precisely the same weak semicontinuity properties of the usual quantiles. 
Further, if $\La$ is strictly decreasing, the left and right $\La$-quantiles always coincide
and are weakly continuous everywhere, making them a remarkable example of a  functional that is weakly continuous throughout $\M$.
\cite{FMP2014}
have shown that if $\La$ is right-continuous, then $Q_\La^+$ is weakly
upper semicontinuous. Notice however that the definition of $Q_\Lambda^+$ in \cite{FMP2014} is based on a supremum and coincides with ours only in the case of a decreasing $\La$. In their Proposition 5, \cite{BPR2017} proved a weak continuity result for continuous $\La$ on a suitably restricted class of CDFs. 
A useful tool to establish these properties is Proposition 2.5 in \cite{FMP2014}, showing that
weak lower semicontinuity of any monotone $\rho \colon \M \to \R$ is equivalent to continuity from above defined by 
\begin{center}$F_n$ weakly convergent to $F$ and $(F_n)$ pointwise decreasing$\quad\implies\quad
\lim_{n\to\infty} \rho(F_n) =\rho(F),
$\end{center}
and similarly weak upper semicontinuity is equivalent to continuity from below, defined analogously.
We will repeatedly use these results, noting that they also hold for $[-\infty,\infty]$-valued functionals. 

Our aim is to provide characterisations of weak semicontinuity properties of left and right $\Lambda$-quantiles in terms of the properties of $\La$. 
The subsequent Theorem~\ref{thm:right-usc} focuses on right $\La$-quantiles.
It shows that weak upper semicontinuity holds under much more general assumptions than previously recognised in the literature. 

\begin{theorem}[Weak upper semicontinuity of $Q_\La^+$]
\label{thm:right-usc}
Let $\La\colon\R\to[0,1]$. 

\begin{enumerate}[label=\tn{(\alph*)}]
    \item $Q_\La^+$ is weakly usc if and only if
\begin{equation}\label{cond:usc} \La(x)\ge \liminf_{y\downarrow x}\La(y), \quad \text{ for all } x \in \R. \end{equation}
 
 \item Suppose $\La$ is additionally right lsc and $Q_\La^+$ is weakly usc. 
 Then $Q_\La^+$ is weakly continuous if and only if 
        there exists a threshold $a\in[-\infty,\infty]$ such that 
        \begin{equation}\label{eq:shape-La}
\La\equiv 1\ \text{on } (-\infty,a) \text{ and is strictly decreasing on } [a,\infty).
\end{equation}
\end{enumerate}
\end{theorem}
\begin{proof}
Regarding statement (a), suppose that \eqref{cond:usc} does not hold, i.e., we find $x_0\in\R$ such that $\lambda:=\liminf_{y\downarrow x_0}\La(y)$ satisfies $d:=\lambda-\La(x_0)>0$. 
    Choose $z>0$ such that $\inf_{x_0<y\le x}\La(y)>\lambda-\tfrac d 2=\La(x_0)+\frac d 2$ holds for all $x\in (x_0,x_0+z]$.
    For $n\in\N$, define 
    \[F_n:=\La(x_0)\ind_{[x_0,x_0+2^{-n}z)}+(\La(x_0)+\tfrac d2)\ind_{[x_0+2^{-n}z,x_0+z)}+\ind_{[x_0+z,\infty)}.\]
    By construction, $\{F_n>\La\}\subseteq[x_0+z,\infty)$, which means that $Q_\La^+(F_n)\ge x_0+z$ for all $n\in\N$.
    Nevertheless, $F_n$ converges weakly to  $F:=\big(\La(x_0)+\tfrac d 2\big)\ind_{[x_0,x_0+z)}+\ind_{[x_0+z,\infty)}$,
    and $Q_\La^+(F)=x_0$.
    This observation excludes upper semicontinuity. 

    Conversely, suppose \eqref{cond:usc} holds and let $(F_n)\subseteq\M$ be a sequence weakly convergent to $F$. 
    If $Q_\La^+(F)=\infty$, there is nothing to show. 
    Else, let $x\in\R$ be such that $F(x)>\La(x)$. 
    By condition \eqref{cond:usc}, 
    $\La(x)\ge \lim_{y\downarrow x}\La(y)$, i.e., $F(x)>\inf_{x<s<y}\La(s)$ holds for all $y>x$ close enough to $x$.
    For all such $y$ fix $s_y\in(x,y)$ such that $F(x)>\La(s_y)$. 
    For all continuity points $c\in (x,s_y)$, 
    \[\liminf_{n\to\infty}F_n(s_y)\ge \lim_{n\to\infty}F_n(c)\ge F(x)>\Lambda(s_y),\]
    meaning that $\limsup_{n\to\infty}Q_\La^+(F_n)\le s_y$. 
    Finally, sending $y\downarrow x$ delivers 
    $\limsup_{n\to\infty}Q_\La^+(F_n)\le x$.
    It remains to choose a sequence of $x$ that converges to $Q_\La^+(F)$ to verify
    $\limsup_{n\to\infty}Q_\La^+(F_n)\le Q_\La^+(F)$.
    This is the upper semicontinuity of $Q_\La^+$.

    For statement (b), assume first that $\La$ is as described by \eqref{eq:shape-La}. 
    If $\La$ is strictly decreasing, then $Q_\La^+$ is finite and known to be weakly continuous. Thus, we only need to consider the case where $a>-\infty$. 
    If $a=\infty$, then $Q_\La^+\equiv\infty$ and thus weakly continuous. 
    Else, if $a<\infty$, we know that $Q_\La^+(F)\ge a$ for all $F\in\M$. 
    Also, in order to verify continuity, we only need to show lower semicontinuity. 
    Let $(F_n)\subseteq\M$ be a sequence weakly convergent to $F$. 
    If $Q_\La^+(F)=a$, there is nothing to show. 
    If $q:=Q_\La^+(F)>a$, $F(x)\le \La(x)$ for all $x\in[a,q)$. Choosing a continuity point $y\in[a,x)$ of $F$, we have 
    $\lim_{n\to\infty}F_n(y)=F(y)\le F(x)\le \La(x)<\La(y)$,
    which means that $Q_\La^+(F_n)\ge y$ for all $n$ large enough. Letting $y\uparrow x$ and $x\uparrow q$, we find 
    $\liminf_{n\to\infty}Q_\La^+(F_n)\ge q$.
    
    Now suppose that $Q_\La^+$ is weakly continuous.   
    Towards a contradiction, assume the existence of $x,y\in\R$ with $x<y$ such that $\La(x)<1$ and $\La(x)\le \La(y)$, and set $m:=\inf_{[x,y]}\La$.

    \textsc{Case 1:} $\La(x)=m$. 
    In this case, the sequence of CDFs 
    \[F_n:=\big[\La(x)+\tfrac 1n(1-\La(x))\big]\ind_{[x,y)}+\ind_{[y,\infty)},\quad n\in\N,\]
    satisfies by construction that $Q_\La^+(F_n)=x$ and converges weakly to 
    $F:=\La(x)\ind_{[x,y)}+\ind_{[y,\infty)}$, but $Q_\La^+(F)\ge y$ must hold, contradicting weak continuity. 
    
    \textsc{Case 2:} $\La(x)>m$.
    Choose a sequence $(z_n)\subseteq[x,y)$ converging to some $z\in[x,y]$ and a null sequence $(\delta_n)\subseteq(0,\infty)$ such that $\La(z_n)$ decreases to $m$ and $\La(z_n)+\delta_n<1$ for all $n\in\N$. 
    The additional assumption of right lower semicontinuity implies 
    \[m<\La(x)\le \La(y)=\liminf_{k\downarrow y}\La(k).\]
    Hence, there exists $\eps>0$ such that $M:=\inf_{[y,y+\eps]}\La>m$. 
    For $n\in\N$ large enough,
    \[F_n:=(\La(z_n)+\delta_n)\ind_{[z_n,y)}+M\ind_{[y,y+\eps)}+\ind_{[y+\eps,\infty)}\]
    is a CDF that 
    satisfies $Q_\La^+(F_n)=z_n$ by construction. Moreover, $F_n$ converges weakly to 
    \[F:=m\ind_{[z,y)}+M\ind_{[y,y+\eps)}+\ind_{[y+\eps,\infty)}.\]  
    By construction, $F(k)\le \La(k)$ holds for all $k<y+\eps$, meaning that 
    \[Q_\La^+(F)\ge y+\eps>y\ge z=\lim_{n\to\infty}Q_\La^+(F_n),\]
    contradicting weak continuity. Hence, also this case must be excluded.
Consequently, weak continuity implies for all $x<y$ that either $\La(x)>\La(y)$ or $\La(x)=\La(y)=1$. 
\end{proof}

Condition~\eqref{cond:usc} is very general, since it requires neither
monotonicity nor any right- or left-regularity of $\La$. It expresses the fact
that $\La$ has \emph{no upward jumps from the right}, and is automatically
satisfied whenever $\La$ is constant, decreasing, or right-continuous; we thus
recover the sufficient conditions known in the literature, substantially
generalising them to an equivalence. 

We now turn to weak semicontinuity properties of left $\Lambda$-quantiles, which always hold in the case of the usual quantiles or of a decreasing $\Lambda$. 
The next theorem provides a full characterisation of weak continuity and weak lower semicontinuity of left $\Lambda$-quantiles under right lower semicontinuity. 

\begin{theorem}[Weak lower semicontinuity of $Q_\La^-$]
\label{thm:left-lsc}
Let $\La \colon \R \to [0,1]$ be right lower semicontinuous and in standard representation such that $Q_\La^->-\infty$, and let
$N_0$ defined by \eqref{eq:N-zero}. Then:
\begin{enumerate}[label=\tn{(\alph*)}]
    \item $Q_\La^-$ is weakly lsc if and only if $\La$ is lower semicontinuous
    \item If $Q_\La^-$ is weakly lsc, then it is weakly continuous if and only if 
    \begin{equation}\label{eq:shape-La 2}
\La \text{ is strictly decreasing on } (-\infty,N_0)
\text{ and } \La \equiv 0 \text{ on } [N_0,\infty),
\end{equation}

\end{enumerate}
\end{theorem}
\begin{proof}
As a preliminary observation, if   there exist $a<b$ such that $\La(a)>0$ is the minimal value that $\La$ attains on $[a,b]$, 
then $Q_\La^-$ is not weakly usc. Indeed, letting
$F_n:=(1-2^{-n})\La(a)\ind_{[a,b)}+\ind_{[b,\infty)}$,
it holds that $Q_\La^-(F_n)=b$ and $F_n\to F=\La(a)\ind_{[a,b)}+\ind_{[b,\infty)}$ weakly as $n\to\infty$. However, $Q_\La^-(F)=a$, contradicting upper semicontinuity.

For statement (a), suppose first that $Q_\La^-$ is weakly lsc. We need to prove that $\La(x)\le\liminf_{y\uparrow x}\La(y)$ holds for all $x\in\R$.
As this holds trivially if  $\La(x)=0$, we assume directly that $\La(x)>0$, meaning that also $\La(y)>0$ for all $y<x$ since $\La$ is in standard representation. 
Let $(x_n)\subseteq(-\infty,x)$ be a sequence increasing to $x$ such that $\La(x_n)$ converges to $\lambda:=\liminf_{y\uparrow x}\La(y)$. 
For each $n\in\N$ define a CDF
$F_n:=\La(x_n)\ind_{[x_n,x+1)}+\ind_{[x+1,\infty)}$.
By construction, $Q_\La^-(F_n)=x_n$, and the sequence converges weakly to 
$F:=\lambda\ind_{[x,x+1)}+\ind_{[x+1,\infty)}$.
By lower semicontinuity of the left $\La$-quantile, $Q_\La^-(F)=x$. By Lemma~\ref{lem:attain}, 
$\lambda=F(x)\ge \La(x)$.

For the converse implication, assume that $\La$ is lsc. 
Suppose $(F_n)\subseteq\mathcal M$ is a pointwise decreasing sequence of CDFs converging weakly to $F\in\mathcal M$, and note that $Q_\La^-(F_n)$ is an increasing sequence. 
Set $q:= \lim_{n\to\infty}Q_\La^-(F_n).$ 
If $q=\infty$, 
then $F(x)\le F_n(x)<\La(x)$ holds for all $x<Q_\La^-(F_n)$. 
Hence, $F(x)<\La(x)$ must hold for all $x\in\R$, meaning that also $Q_\La^-(F)=\infty$.
Now suppose $q<\infty$. 
By our assumption on $\La$ and Lemma~\ref{lem:attain}, we have
\[\La(Q_\La^-(F_n))\le F_n(Q_\La^-(F_n))\le F_n(y),\]
where $y>q$ is an arbitrary continuity point of $F$. 
Moreover, lower semicontinuity of $\La$ delivers 
\[\liminf_{n\to\infty}\La(Q_\La^-(F_n))\ge\La(q).\]
Hence, 
\[\La(q)\le\liminf_{n\to\infty}\La(Q_\La^-(F_n))\le \lim_{n\to\infty}F_n(y)=F(y).\]
Letting $y\downarrow q$ and using that $F$ is right-continuous,
$\La(q)\le F(q)$, meaning that $Q_\La^-(F)\le q$.  
The converse inequality holds because of monotonicity of $Q_\La^-$.

For (b), suppose $\La$ is not of shape~\eqref{eq:shape-La 2}.
By right lower semicontinuity and $\La$ being in standard representation, 
we can find $a<b$ such that $\La(a)\le \La(b)$ and $\La(b)>0$.
Because of lower semicontinuity, $\La$ attains its minimum $m$ at some $t\in[a,b]$. 
Note that $m>0$ because $\La$ is in standard representation, meaning that 
$m=0$ would entail $\La(b)=0$.
If we can choose $t=b$, then $\La(a)=\La(b)=m$ and $Q_\La^-$ is not weakly upper semicontinuous by the initial observation. 
Else, we can choose $t\in[a,b)$ and consider the interval $[t,b]$, again arriving at the conclusion that $Q_\La^-$ is not weakly upper semicontinuous. 
Conversely, assume $\La$ is of shape \eqref{eq:shape-La 2}.
If $N_0=\infty$, the claimed implication follows from \cite[Corollary 2.9]{BP}. 
Else, since $Q_\La^-$ is weakly lsc by statement (a),  
we only need to prove that it satisfies weak upper semicontinuity.
Suppose $(F_n)\subseteq\M$ is a sequence of CDFs converging weakly to $F$ and that $q:=\limsup_{n\to\infty}Q_\La^-(F_n)>-\infty$.
For all $x<q$ and infinitely many $n$, $F_n(x)<\La(x)$, meaning that $x\le N_0$. If $y\in(x,q)$ is a continuity point of $F$,  
$F(x)\le F(y)=\lim_{n\to\infty}F_n(y)\le \La(y)<\La(x)$.
As $\La$ is decreasing, $Q_\La^-(F)\ge x$. Letting $x\uparrow q$ yields $Q_\La^-(F)\ge q$,
which is upper semicontinuity. Together with the lower
semicontinuity established in (a), this concludes the proof.
\end{proof}

Notice that if $\La$ is constant or decreasing then it is \emph{left} lower semicontinuous, so from the assumptions of right lower semicontinuity it follows that $\La$ is lower semicontinuous, implying the validity of (a) and recovering the known results mentioned at the beginning of this section. 
If $\La$ is positive, item~(b) shows that weak continuity is equivalent to $\La$
being strictly decreasing, in which case $Q_\La^-$ and
$Q_\La^+$ coincide.

From the statistical point of view, weak continuity is the property that
makes a law-invariant functional well behaved on data. Since the empirical
distribution $\widehat F_n$ converges weakly to $F$ almost surely by the Glivenko-Cantelli theorem, weak
continuity of $Q_\La^{\pm}$ at $F$ yields the strong consistency of the empirical
plug-in estimator and is the analytic
counterpart of qualitative robustness in the sense of \cite{Hampel}.
In the framework of
\citet{CDS2010}, a \emph{risk measurement procedure} combines the estimation of the loss
distribution with the evaluation of a risk functional, and continuity with respect to
weak convergence is what makes the historical estimator qualitatively robust. 
These properties
were thoroughly discussed for $\La$VaR in \cite{BPR2017}; our characterisations above
delineate exactly which $\La$ enjoy them, beyond the monotone and
right-continuous setting considered there.
Further, weak lower
semicontinuity coincides here with the \emph{prudence} axiom of \citet{WZ2021}, i.e., 
the requirement that approximating a financial position with respect to a.s. convergence does not underestimate
its risk, and with the so-called \emph{super Fatou property} of \citep[Remark~3.7]{GMX2020}, an equivalence later studied in depth by \citet{AL2024}. 
Theorem~\ref{thm:left-lsc} therefore identifies $Q_\La^-$ as a prudent risk measure
whenever $\La$ is lower semicontinuous, whereas for $Q_\La^+$ prudence is compatible with weak
upper semicontinuity only in the rigid situation described in
Theorem~\ref{thm:right-usc}(b).

\section{Convexity of the level sets}\label{sec:CxLS}

In this section we focus on the property of convexity of level sets with respect to mixtures, called the CxLS property in the literature.   
Recall that a functional $\rho$ on $\M$ has the CxLS property if 
\[
\rho(F)=\rho(G)=c \implies \rho(\lambda F+(1-\lambda)G)=c,
\]
for all $\lambda \in (0,1)$ and $c\in\R$. 
It has been known since \cite{O85} that the CxLS property is a necessary condition for elicitability, that in turn is the property of being the minimiser of an expected loss; for this reason the CxLS property of risk measures and of related functionals has been extensively studied in the literature; we refer e.g. to \cite{BB2015}, \cite{DBBZ2016}, \cite{Z2016}, \cite{WW2020}, and the references therein. The classical quantiles enjoy the CxLS property, and for general Lambda-quantiles \cite{BP} showed that in the case of a decreasing $\La$, the CxLS property always holds; see also \cite{BB2015}.  
Previously, \cite{BPR2017} showed by means of a counterexample that right Lambda-quantiles may not have the CxLS property, and gave further sufficient conditions. We build on their ideas 
in order to try to come as close as possible to an if-and-only-if characterisation. Throughout the section we impose the condition that Lambda-quantiles are finite on all of $\M$ as in Lemma~\ref{lem:liminf}.  

\subsection{The CxLS property of the left \texorpdfstring{$\La$}{Lambda}-quantile}

As a first result, we show that under the attainment condition the CxLS property always holds. 

\begin{proposition}[CxLS property of $Q_\La^-$]
\label{prop:CxLS}
Let $\La$ be right lower semicontinuous. Then $Q_\La^-$ has the CxLS property. 
\end{proposition}
\begin{proof}
First, if $H\in\M$ satisfies $Q_\La^-(H)=c$ and permits to choose a strictly decreasing sequence $(x_n)\subseteq(c,\infty)$ along which $H(x_n)\ge\La(x_n)$, we obtain 
\[H(c)=\lim_{n\to\infty}H(x_n)\ge\liminf_{y\downarrow c}\La(y)\ge \La(c)\]
by right lower semicontinuity. 
Hence, $Q_\La^-(H)=c$ if and only if $H(c)\ge\La(c)$ and $H(x)<\La(x)$ for all $x<c$. If $Q_\La^-(F)=Q_\La^-(G)=c$ and $\lambda\in(0,1)$, then also
$\lambda F(c)+(1-\lambda)G(c)\ge \Lambda(c)$ and $\lambda F(x)+(1- \lambda)G(x)<\Lambda(x)$ for all $x <c$,
which means that $Q_\La^-( \lambda F+(1- \lambda)G)=c$. 
\end{proof}

Despite its generality, this condition is only sufficient for the CxLS property. 
While it does not seem possible in full generality, we will show a simple if-and-only-if condition for a large class of $\La$ described by the following definition.

\begin{definition}
[Local right monotonicity]
\label{def:LRM} 
We say that $\La \colon \R \to [0,1]$ is \emph{locally right monotone} if for all $x\in\R$, there exists $y>x$ such that $\La|_{(x,y)}$ is monotone. 
\end{definition}

This assumption essentially requires that $\La$ does not have infinite oscillations in any right neighbourhood of any $x \in \R$. It immediately implies that $\La$ is right regular. Since there is no control on oscillations on left neighbourhoods, it does not entail that $\La$ is of bounded variation. Also, it is not implied by $\La$ having bounded variation, because there can be functions with infinite oscillations and bounded variation. 
Appendix~\ref{sec:ass} studies this class of functions in more detail. The following lemma is instrumental for the construction of a counterexample of the CxLS property. 

\begin{lemma}\label{lem:downward seq}
    Suppose $\La$ is locally right monotone and let
\begin{equation}
\label{eq:def-D}
\mathcal{D} : = \{ x \in \R \mid \text {there exists } y > x \text{ such that } \La \text { is decreasing on } (x,y)\}.   
\end{equation}
Let $x\notin\mathcal D$. Then there is a sequence $x_n \downarrow x$
such that the sequence $\La(x_n)$ is strictly decreasing.
\end{lemma}
\begin{proof}
    By local right monotonicity, since $x \not \in \mathcal{D}$ there exists $y>x$ such that $\La|_{(x,y)}$ is increasing. In particular, $\La(x+)=\inf_{(x,y)}\La$.
    Moreover, if there were  $z\in(x,y)$ such that $\sup_{(x,z)}\La=\La(x+)$ would hold, then $\La|_{(x,z)}$ would be constant. This would lead to the contradiction that $x\in\mathcal D$. 
    Next, suppose $z\in(x,y)$ satisfies $\La(z)>\La(x+)$, i.e., $\La$ is not constant on $(x,z)$. 
    If $\La$ would only attain two values on $(x,z)$, we would again arrive at the contradictory conclusion that $x\in\mathcal D$. Hence, $\La$ must attain at least three values.
    Taking all these observations together allows us to iteratively construct the sequence whose existence is claimed. 
\end{proof}

\begin{theorem}[CxLS property of $Q_\La^-$]
\label{thm:CxLS}
Suppose $\Lambda$ is in standard representation and locally right monotone.
Let $\mathcal D$ be as in \eqref{eq:def-D}. 
The following are equivalent. 
\begin{enumerate}[label=\tn{(\alph*)}]
    \item $Q_\La^-$ has the CxLS property.
    \item For all $x\notin \mathcal D$, it holds that $\La(x)\le\La(x+)$.
\end{enumerate}
\end{theorem}
\begin{proof} 
(a) implies (b): Assume that (b) fails, i.e., we find $a\notin\mathcal D$ such that $\La(a)>\La(a+)$. 
By Lemma~\ref{lem:downward seq}, we find $b>a$ and a sequence $(x_n)\subseteq(a,b)$ strictly decreasing to $a$ such that $\La|_{(a,b)}$ is increasing and 
$\Lambda(x_n)>\Lambda(x_{n+1})$, for all $n\in\N.$ 
Let $F$ be the piecewise constant CDF with range
$\{0,1,\La(a+),\La(x_1),\La(x_3),\La(x_5),\dots\}$ such that $Q_1^-(F)=b$, such that
$Q_{\La(x_{2i-1})}^-(F)=x_{2i-1}$ for all $i\in\N$, and such that $F(a)=\La(a+)$ and
$F(x)=0$ for $x<a$.
Similarly, let $G$ be the piecewise constant CDF with range
$\{0,1,\La(a+),\La(x_2),\La(x_4),\La(x_6),\dots\}$ such that $Q_1^-(G)=b$, such that
$Q_{\La(x_{2i})}^-(G)=x_{2i}$ for all $i\in\N$, and such that $G(a)=\La(a+)$ and
$G(x)=0$ for $x<a$.
By construction, $Q_\La^-(F)=Q_\La^-(G)=a$.
Nevertheless, 
    $\frac 1 2F(x)+\frac 1 2 G(x)<\Lambda(x_1)\le\Lambda(x)$ for $x\in[x_1,b)$, $\frac 1 2F(x)+\frac 1 2 G(x)<\Lambda(x_2)\le\Lambda(x)$ for $x\in[x_2,x_1)$, and so on. 
    Also, $\frac 1 2F(a)+\frac 1 2G(a)<\Lambda(a)$ and $\Lambda(x)>\frac 1 2F(x)+\frac 12G(x)=0$ for $x<a$. 
    Hence, 
    $Q_\Lambda^-(\tfrac 1 2 F+\tfrac 1 2 G)=b>a.$

    (b) implies (a): Suppose $Q_\La^-(F)=Q_\La^-(G)=c$. 
    Suppose $c\notin\mathcal D$, so $\La(c)\le \La(c+)$. One easily sees that this necessitates $\La(c)\le \min\{F(c),G(c)\}$ and $\La(x)>\max\{F(x),G(x)\}$ for $x<c$. 
    Argue like in the proof of Proposition~\ref{prop:CxLS}.
    If $c\in\mathcal D$, then $Q_\La^-(H)=c$ holds for $H\in\M$ if and only if at least one of the following holds: 
    \begin{enumerate}[label=(\arabic*)]
        \item $H(c)\ge \La(c+)$, in which case $H\ge \La$ on a nonempty interval of shape $(c,d)$;
        \item $H(c)\ge \La(c)$. 
    \end{enumerate}
    If $F$ and $G$ both fall in category (1) or (2), so does any of their mixtures. 
    Now suppose that $F$ is of type (1), but not of type (2), and $G$ is of type (2). Then $\La(c+)\le F(c)<\La(c)\le G(c)$,
    meaning that $G$ must also be of type (1). 
    If $F$ is of type (1) and $G$ is of type (2), but not of type (1), then 
 $F(c)\ge\La(c+)>G(c)\ge \La(c)$. 
    This means that both of them must be of type (2).   
\end{proof}

\subsection{The CxLS property of the right \texorpdfstring{$\La$}{Lambda}-quantile}

As shown in \cite[Example 15]{BPR2017}, $Q_\Lambda^+$ might not have the CxLS property if $\La$ is increasing and right-continuous.  
In this section, we will refine this observation and prove a counterpart to Theorem~\ref{thm:CxLS}, again working under local right monotonicity. 

\begin{proposition}[CxLS property of $Q_\La^+$]
\label{prop:CxLS+}
Suppose $\Lambda$ is locally right monotone
and let $\mathcal D$ be as in \eqref{eq:def-D}. 
The following are equivalent. 
\begin{enumerate}[label=\tn{(\alph*)}]
    \item $Q_\La^+$ has the CxLS property.
    \item For all $x\notin \mathcal D$, $\La(x)<\La(x+)$.
\end{enumerate}
\end{proposition}
\begin{proof}
(a) implies (b): Suppose condition (b) fails, i.e., there is $a \not \in\mathcal D$ such that $\La(a)\ge\La(a+)$. 
By Lemma~\ref{lem:downward seq}, we may select $b>a$ and a sequence $(x_n)\subseteq(a,b)$ decreasing to $a$ such that $\La$ is increasing on $(a,b)$ and $(\La(x_n))$ is a strictly decreasing sequence. 
Moreover, select a sequence $(s_n)\subseteq(0,\infty)$ such that, for all $n\in\N$, 
$
\La(x_{n+1})+s_{n+1}<\La(x_{n})-s_{n}.
$
Next, define $F$ and $G$ by
\[F(x)=\begin{cases}
    0&\text{if }x<a,\\[-0.7ex]
    \La(x+)&\text{if }x=a,\\[-0.7ex]
    \La(x_n)+(-1)^ns_n&\text{if }x\in[x_n,x_{n-1}),\,n\ge 2,\\[-0.7ex]
    1&\text{if }x\ge x_1,
\end{cases} 
\] 
\[
G(x)=\begin{cases}
    0&\text{if }x<a,\\[-0.7ex]
    \La(x+)&\text{if }x=a,\\[-0.7ex]
    \La(x_n)-(-1)^ns_n&\text{if }x\in[x_n,x_{n-1}),\,n\ge 2,\\[-0.7ex]
    1&\text{if }x\ge x_1.
\end{cases}\]
By construction, $F(x_n)>\La(x_n)$ if $n$ is even, and $G(x_n)>\La(x_n)$ if $n$ is odd, which results in $Q_\La^+(F)=Q_\La^+(G)=a$. 
However, 
\[(\tfrac 12 F+\tfrac 12 G)(x)=\begin{cases}
    0&\text{if }x<a,\\[-0.7ex]
    \La(x+)&\text{if }x=a,\\[-0.7ex]
    \La(x_n)&\text{if }x\in[x_n,x_{n-1}),\,n\ge 2,\\[-0.7ex]
    1&\text{if }x\ge x_1.
\end{cases}\]
As $\La$ is increasing on $(a,b)$ and $\La(a)\geq \La(a+)$, the set $\{\frac 12 F+\frac 1 2G>\La\}$ is a subset of $[x_1,\infty)$, meaning that the CxLS property is violated. 

(b) implies (a): Suppose $Q_\La^+(F)=Q_\La^+(G)=c$. 
    If $c\notin\mathcal D$, $\La(c)<\La(c+)$ holds by (b). 
    This necessitates $\La(c)<\min\{F(c),G(c)\}$ and $\La(x)\ge \max\{F(x),G(x)\}$ for $x<c$. 
    Consequently, 
    $Q_\La^+(\lambda F+(1-\lambda)G)=c$ holds for every $\lambda\in(0,1)$. 

    Now suppose that $c\in\mathcal D$ and that $H\in\M$ satisfies $Q_\La^+(H)=c$. Then one of the following two cases must apply:  
    \begin{enumerate}[label=(\arabic*)]
        \item $H>\La$ on a nonempty open interval $(c,d)$;
        \item $H(c)>\La(c)$. 
    \end{enumerate}
    If $F$ and $G$ both fall in category (1) or (2), so does any of their mixtures.   
    If $F$ is of type (1), but not of type (2), and $G$ is of type (2), then $\La(c+)\le F(c)\le \La(c)<G(c)$. Given that $\La$ is decreasing on $(c,d)$ if $d$ is suitably chosen, $G$ must also be of type (1). 
    If $F$ is of type (1), and $G$ is of type (2), but not of type (1), then $\La(c)<G(c)\le \La(c+)\le F(c)$. Hence, $F$ must also be of type (2). 
\end{proof}

\subsection{Elicitability}

The distributions constructed in the proofs of Theorem~\ref{thm:CxLS} and
Proposition~\ref{prop:CxLS+} cross $\La$ infinitely often in every right
neighbourhood of their $\La$-quantile value. This suggests that the obstruction to the
CxLS property, and hence to elicitability, disappears once one restricts to
distributions crossing $\La$ only once, which is the domain considered in
\cite{BPR2017}. We conclude the section by showing that this is indeed the case,
for an arbitrary measurable $\La$. Since $Q_\La^-$ and $Q_\La^+$ are
single-valued, we adopt the corresponding definition of elicitability; we refer
to \citet{FFHR21} for the set-valued case.

\begin{definition}[Elicitability]\label{def:elicit}
A map $S\colon\R\times\R\to[0,\infty)$ is a \emph{scoring function}, and a functional
$\rho\colon\M'\to\R$ defined on $\M'\subseteq\M$ is \emph{elicitable} on $\M'$, if there
exists a scoring function $S$ such that
\[
\rho(F)=\underset{x \in \R}{\operatorname{arg\,min}}\int_\R S(x,y)\,dF(y),\qquad F\in\M',
\]
and the minimiser is unique for every $F\in\M'$.
\end{definition}

Following \cite{BPR2017}, we set
\begin{equation}\label{eq:M-La}
\M_\La:=\big\{F\in\M \mid \exists\,\bar x\in\R \text{ s.t. } F<\La \text{ on }(-\infty,\bar x)
\text{ and } F>\La \text{ on } (\bar x,\infty)\big\},
\end{equation}
and we write $\M_\La^1$ for the distributions in $\M_\La$ with finite first moment.
Note that on $\M_\La$ it holds that
$Q_\La^-(F)=Q_\La^+(F)=\bar x$. 
The natural scoring function for $\La$-quantiles has been introduced in \cite{BB2015} and is given by
\begin{equation}\label{eq:score}
S(x,y):=(y-x)^--\int_y^x\La(t)\,dt,\qquad x,y\in\R,
\end{equation}
which for $\La\equiv\alpha$ reduces to the well known pinball loss
$\alpha(y-x)^++(1-\alpha)(y-x)^-$ of the classical quantile.
The following result generalizes \citep[Theorem 8]{BPR2017}, and in particular applies to every $\La$ of bounded variation, the class
considered in Sections~\ref{sec:MR} and \ref{sec:OC}. 
\begin{theorem}[Elicitability of $\La$-quantiles]\label{thm:elicit}
Let $\La\colon\R\to[0,1]$ be measurable. Then $S$ in \eqref{eq:score} is a scoring
function and, for every $F\in\M_\La^1$,
\[
Q_\La^-(F)=Q_\La^+(F)=\underset{x\in\R}{\operatorname{arg\,min}}\int_\R S(x,y)\,dF(y).
\]
In particular, $Q_\La^\pm$ is elicitable on $\M_\La^1$.
\end{theorem}
\begin{proof}
Nonnegativity of $S$ follows from $S(x,y)=\int_y^x(1-\La(t))\,dt\ge0$ for $y\le x$ and
$S(x,y)=\int_x^y\La(t)\,dt\ge0$ for $y>x$, using $0\le\La\le1$. The same computation
gives the representation
\[
S(x,y)=\int_y^x\big(\ind_{\{y\le t\}}-\La(t)\big)\,dt ,
\]
from which we get $S(x,y)-S(x',y)=\int_{x'}^{x}\big(\ind_{\{y\le t\}}-\La(t)\big)\,dt$ for all
$x,x'\in\R$. 
Fix $F\in\M_\La^1$ and let $\bar x=Q_\La^\pm(F)$. As $|S(x,y)|\le2|x-y|$, both expected
scores below are finite, and since the integrand is bounded by $1$ in absolute value,
Fubini's theorem applies and gives for every $x\in\R$
\begin{align}\label{eq:fubini}
\int_\R S(x,y)\,dF(y)-\int_\R S(\bar x,y)\,dF(y)
&=\int_\R\left(\int_{\bar x}^{x}\big(\ind_{\{y\le t\}}-\La(t)\big)\,dt\right)dF(y)
\nonumber\\
&=\int_{\bar x}^{x}\left(\int_\R\big(\ind_{\{y\le t\}}-\La(t)\big)\,dF(y)\right)dt
\nonumber\\
&=\int_{\bar x}^{x}\big(F(t)-\La(t)\big)\,dt.
\end{align}
If $x>\bar x$, then $F-\La>0$ on $(\bar x,x)$ by \eqref{eq:M-La} and the right-hand side
of \eqref{eq:fubini} is strictly positive; if $x<\bar x$, then $F-\La<0$ on
$(x,\bar x)$ and the right-hand side equals $-\int_x^{\bar x}(F-\La)(t)\,dt>0$. Hence
$\bar x$ is the unique minimiser.
\end{proof}
\begin{remark}[Elicitability and the CxLS property]\label{rem:elicit-CxLS}
Since elicitability implies the CxLS property, Theorem~\ref{thm:elicit} gives the
CxLS property of $Q_\La^\pm$ on $\M_\La$ for an arbitrary measurable $\La$.
On the full domain $\M$ the implication can be used in the opposite direction:
whenever condition (b) of Theorem~\ref{thm:CxLS} fails, no scoring function
elicits $Q_\La^{-}$ on $\M$, and similarly for $Q_\La^{+}$ and
Proposition~\ref{prop:CxLS+}. 
\end{remark}
\begin{remark}[Alternative scoring functions]\label{rem:osband}
The scoring function \eqref{eq:score} is far from unique. The map
$V(x,y):=\La(x)-\ind_{\{y\le x\}}$ is an identification function for $Q_\La^\pm$, since
$\int_\R V(x,y)\,dF(y)=\La(x)-F(x)$ vanishes exactly where $F$ meets $\La$. Osband's
principle \citep{O85, FZ2016} then produces the family
\[
S_g(x,y)=\big(g(x)-g(y)\big)^+-\int_y^x\La(t)\,dg(t),\qquad g \text{ strictly increasing},
\]
which reduces to \eqref{eq:score} for $g=\mathrm{id}$. 
The proof of Theorem~\ref{thm:elicit} applies verbatim to $S_g$ and the family $\M_\La^g$ of distributions for which $\int_\R|g(t)|dF(t)$ is finite. The generality of $g$ allows to encompass supersets of $\M^1_\La$. 
\end{remark}

\section{Mixture representations}\label{sec:MR}

In this section we investigate what we call \emph{mixture representations} of Lambda-quantiles. The prototypical result is Corollary \ref{cor:MV} below, that shows that the $\La$-quantile of a distribution $F$ with respect to a decreasing $\La$ can be represented as the usual quantile of a mixture of $F$ and a fixed distribution $G$ with a fixed weight $\lambda$. 
This interesting result has been communicated in \cite{W2024}, where the remarkable connection with the theorem of \cite{Moulin} characterising voting schemes 
that are anonymous, efficient and strategy-proof as \emph{generalised medians} has been pointed out.

In this section we prove a generalised version of this result, that requires only that $\La$ is of bounded variation and satisfies the attainment condition. 
Under these assumptions, $\La$-quantiles can be represented as $\La_{\tn{inc}}$-quantiles with an \emph{increasing} $\La_{\tn{inc}}$ of suitable mixtures with a fixed reference distribution $G$ and with a fixed weight $\lambda$. 
This shows how the mixture representation \emph{reduces the complexity} of the class of $\Lambda$ involved from decreasing to constant, in the case of Corollary \ref{cor:MV} below, and from bounded variation to increasing in the generalised case formalised in the next theorem. 
\begin{theorem}[Mixture representation]
\label{thm:MW-1}
Let $\La \colon \R \to [0,1]$ have bounded variation and be right lower semicontinuous. 
Then there exist an increasing function $\La_{\rm{inc}} \colon \R \to [0,1]$, a fixed weight $\lambda\in(0,1]$ with $\lambda \geq \La_{\rm inc}(-\infty)$, and a fixed distribution $G\in\M$ such that 

\begin{equation}\label{eq:rep+-}Q_\La^\pm(F)=Q_{\La_{\rm{inc}}}^\pm(\lambda F+(1-\lambda)G), \quad \text { for all }F\in\M.\end{equation}
If $\La$ is in standard representation, then so is  $\La_{\rm{inc}}$. 
\end{theorem}

\begin{proof}
Since $\La$ has bounded variation, Lemma~\ref{lem:BV}(a) yields the existence of the  limit $a:=\La(-\infty)$. 
Items (c) and (e) in the same result yield increasing and bounded functions 
$G^{\pm} \colon \R \to [0,+\infty)$, $G^-$ being right-continuous, such that 
\begin{equation}
\label{eq:dec}
\La=a+G^+-G^-. 
\end{equation}
Let $b:=G^+(\infty) \in \R$ and $c:=G^-(\infty) \in \R$. 
As $\La$ takes values in $[0,1]$, it holds that
\begin{equation}
\label{abc}
\lim_{x \to \infty}
\La(x)
=a+b-c \leq 1.
\end{equation} 
We distinguish two cases. If $c=0$, then $G^-\equiv 0$, so $\La = a + G^+$ is already increasing, and the thesis is obtained by choosing $\La_{\rm{inc}}=\La$, $\lambda = 1$ and any $G \in \M$. If $c>0$, we set $G:=\frac 1 cG^-$, and from the properties of $G^-$ it follows that $G\in\M$. 
In general, from the representation \eqref{eq:dec} it follows that
\begin{equation*}
Q_\Lambda^+(F)=\inf\{x\in\R\mid F(x)+G^-(x)>a+G^+(x)\},
\end{equation*}
so we get
\[
Q_\Lambda^+(F)=\inf\{x\in\R\mid F(x)+cG(x)>a+G^+(x)\}.
\]
Letting
$\lambda=1/(1+c)$ and $\Lambda_{\rm{inc}}=(a+G^+)/(1+c)$ gives formula \eqref{eq:rep+-}. Further, from \eqref{abc} we check 
\[
0\le \frac{a}{1+c}\le \La_{\rm{inc}}\le \frac{a+b}{1+c}\le 1.
\]
Moreover, 
$\La_{\rm{inc}}(-\infty)=a/(1+c)\le 1/(1+c)=\lambda.
$
The proof of the statement in case of the left-quantile is similar,
replacing the strict inequality by ``$\ge$'' in the identities. 
It remains to show that $\La_{\rm inc}$ is in standard representation. If $\La\equiv 0$, we may choose 
$a=0$ and $G^+\equiv 0$, so $\La_{\rm inc}\equiv 0$, which is in standard representation. 
Otherwise since $\La$ is in standard representation then $\La(x)>0$ for all $x<N_0$, hence 
\[
\La_{\rm inc}(x)=\frac{a+G^+(x)}{1+c}\ge\frac{\La(x)}{1+c}>0,\qquad x<N_0.
\]
Since $\La_{\rm inc}$ is increasing, it follows that 
$\La_{\rm inc}>0$ on all of $\R$, so $\La_{\rm inc}$ 
is in standard representation. 
\end{proof}

\begin{remark}
It is easy to see that the  representing triple $(\La_{ \rm inc}, \lambda, G)$ is in general not unique. 
For example, let $\La\equiv\alpha\in(0,1)$, so that $Q_\La^{\pm}=Q_\alpha^{\pm}$ is the 
classical $\alpha$-quantile. 
Let $G$ be any CDF, and let $\lambda\in(0,1)$.
Then 
$\La_{\rm inc}=\alpha\lambda+(1-\lambda)G$ is increasing, and for any such 
choice,
\[
Q_{\La_{\rm inc}}^\pm(\lambda F+(1-\lambda)G)
=\inf\{x\in\R\mid \lambda F(x)+(1-\lambda)G(x)\gtge\alpha\lambda+(1-\lambda)G(x)\}
=Q_\alpha^\pm(F),
\]
showing non-uniqueness.
\end{remark}

\begin{corollary}[Decreasing case]\label{cor:MV}
For a decreasing $\La\colon\R\to[0,1]$, there exist $\lambda\in[\frac 1 2,1]$, $0\le\alpha\le\lambda$, and $G\in \mathcal M$ such that, for all $F\in\M$, 
\begin{equation}\label{eq:repMW}Q_\La^{\pm}(F)=Q_\alpha^\pm(\lambda F+(1-\lambda)G).\end{equation}
Moreover, the following statements hold:
    \begin{enumerate}[label=\tn{(\alph*)}]
        \item If $\La(\infty)=0$, we can choose $\lambda\ge\frac 1 2$ and $\alpha\le\frac 1 2$. 
        \item  If $\Lambda(-\infty)=1$, we can choose $\lambda=\alpha\ge\frac 1 2$.
        \item If $\La(\infty)=0$ and $\La(-\infty)=1$, we can choose $\alpha=\lambda=\frac 1 2$, i.e., the left (right) $\La$-quantile is the median of an equal-weights mixture with the fixed distribution $G$. 
        \item If $\Lambda$ is positive and not constant, then $\alpha$, $\lambda$ and $G$ are unique.
    \end{enumerate}
\end{corollary}

\begin{proof}
Since $\La$ is decreasing, we can assume without loss of generality that $\La$ is right-continuous. For the left quantile, $\La$ is then also in standard representation.
Apply Theorem~\ref{thm:MW-1} to $\Lambda=\La(-\infty)-\big(\La(-\infty)-\La \big)$. 
    In the proof of Theorem~\ref{thm:MW-1} choose $G^+=0$ and $G^-=\La(-\infty)-\La$. 
    If $G^-=0$, set $\lambda=1$, $\alpha=\La(-\infty)$, and let $G$ be arbitrary. Else, set 
    \[c:=\La(-\infty)-\La(\infty),\quad\lambda:=\frac 1{1+c},\and \alpha:=\frac{\La(-\infty)}{1+c}.\]
    Clearly, $\lambda\ge \alpha$. 
    The additional assertions (a)--(c) follow easily. 
    For (d), we focus first on the left $\La$-quantile and call triplets $(\alpha,\lambda,G)$ such that $\rho=Q_\La^-$ can be represented via \eqref{eq:repMW} feasible.
As 
\[Q_\La^-(F)=\inf\{x\in\R\mid F(x)\ge\tfrac 1{\lambda}(\alpha-(1-\lambda)G)\}=\inf\{x\in\R\mid F(x)\ge\tfrac 1{\lambda}(\alpha-(1-\lambda)G)_+\}\]
the uniqueness statement implicit in Proposition~\ref{prop:geq} delivers 
\[\La(x)=\frac 1 \lambda\big(\alpha-(1-\lambda)G(x)\big)_+,\]
no matter which feasible triplet we choose. 
Moreover, as $\La>0$, $G(x)<\alpha/(1-\lambda)$ must hold for all $x\in\R$. 
Taking the limits $x\to -\infty$  and $x\to\infty$ delivers 
\[\La(-\infty)=\frac\alpha\lambda\and\La(\infty)=\frac{\alpha-1+\lambda}\lambda.\]
This permits to express $\alpha$ and $\lambda$ in terms of the unique quantities $\La(\infty)$ and $\La(-\infty)$ as
\[\alpha=\frac{\La(-\infty)}{1+\La(-\infty)-\La(\infty)},\quad \lambda=\frac{1}{1+\La(-\infty)-\La(\infty)}.\]
Also, $G$ can then be uniquely inferred from $\La$. 
For the right $\La$-quantile the argument is the same, invoking Proposition~\ref{prop:geq-2} 
in place of Proposition~\ref{prop:geq}, once we know that $\theta:=\tfrac 1{\lambda}(\alpha-(1-\lambda)G)$ 
is nonnegative, so that the truncation above may again be omitted. 
Indeed, if $\theta(x_0)<0$ for some $x_0\in\R$, then $F(x_0)\ge 0>\theta(x_0)$ for every $F\in\M$, 
whence $Q_\La^+(F)\le x_0$. As $\La$ is decreasing and not constant, $\La(\infty)<\La(-\infty)\le 1$, 
so we may pick $y>x_0$ with $\La<1$ on $[y,\infty)$; positivity of $\La$ then yields 
$Q_\La^+(B_{y,0})=y>x_0$, a contradiction.
\end{proof}

The following proposition provides a kind of converse of Theorem \ref{thm:MW-1} above. 

\begin{proposition}[Closure of $\La$-quantiles under mixing]
\label{prop:clos-mix}
Let $\Lambda_0\colon\R\to\R$ have bounded variation, $\lambda\in(0,1]$ and $G\in\M$ be such that 
\begin{equation}\label{eq:mixturecond}(1-\lambda)G\le\Lambda_0\le\lambda+(1-\lambda)G.\end{equation}
Then $\La:=\tfrac1\lambda\big(\Lambda_0-(1-\lambda)G\big)$ is $[0,1]$-valued, has bounded 
variation, and it holds that
\begin{equation}\label{eq:closure-gen}
Q_{\Lambda_0}^{\pm}\big(\lambda F+(1-\lambda)G\big)=Q_\La^{\pm}(F), \qquad F \in \M.
\end{equation}
Moreover, if $\Lambda_0$ is right lower semicontinuous, then so is $\La$. 
\end{proposition}

\begin{proof}
By \eqref{eq:mixturecond} we have $\lambda\La=\Lambda_0-(1-\lambda)G\ge 0$ and 
$\lambda\La\le\lambda$, so that $\La$ takes values in $[0,1]$. Being a linear combination 
of two functions of bounded variation, $\La$ has bounded variation as well, with 
$\TV(\La)\le\tfrac1\lambda\big(\TV(\Lambda_0)+1-\lambda\big)$.
Fix now $F\in\M$ and set $H:=\lambda F+(1-\lambda)G$, which again belongs to $\M$. 
Since $\lambda>0$, for every $x\in\R$
\[
H(x)\ge\Lambda_0(x)\quiff \lambda F(x)\ge\Lambda_0(x)-(1-\lambda)G(x)\quiff F(x)\ge\La(x),
\]
and the same chain of equivalences holds with ``$\ge$'' replaced by ``$>$''. Consequently 
\[
\{H\ge\Lambda_0\}=\{F\ge\La\}\and\{H>\Lambda_0\}=\{F>\La\},
\]
and taking infima of these sets, with the convention $\inf\varnothing=\infty$, delivers 
\eqref{eq:closure-gen} for $Q^-$ and for $Q^+$ respectively, both sides being understood 
as elements of $[-\infty,\infty]$.

Finally, suppose that $\Lambda_0$ is right lower semicontinuous and let $x\in\R$. 
As $G$ is right-continuous, $\lim_{y\downarrow x}(1-\lambda)G(y)=(1-\lambda)G(x)$, from which
\[
\liminf_{y\downarrow x}\lambda\La(y)
=\liminf_{y\downarrow x}\Lambda_0(y)-(1-\lambda)G(x)
\ge\Lambda_0(x)-(1-\lambda)G(x)=\lambda\La(x).
\]
Dividing by $\lambda>0$ shows that $\La$ is right lower semicontinuous. 
\end{proof}

Thus mixing the argument $F$ with a fixed distribution $G$ at a fixed weight maps the 
class of $\La$-quantiles with $\La$ of bounded variation and right lower semicontinuous 
into itself, provided \eqref{eq:mixturecond} holds. The induced parameter $\La$ 
inherits the decreasing behaviour of $-(1-\lambda)G$, so it need not be monotone even 
when $\La_0$ is. This gives a second reason to admit non-monotone $\La$, alongside the one 
from minimum aggregation in Section~\ref{sec:aggr}, where a pointwise minimum of 
monotone functions need not be monotone.

\section{Ordinal covariance}\label{sec:OC}
In this section, we consider functionals defined both on the set $\M$ and on the space $\CL^0$ of all real-valued random variables over an underlying atomless probability space $(\Omega,\CF,\P)$. If the functional $\rho\colon \CL^0\to\R$ in question is {\em law invariant}, i.e., $\rho(X)=\rho(Y)$ whenever the CDFs of $X$ and $Y$ under $\P$ coincide, these two perspectives are perfectly interchangeable. Indeed, each $\rho\colon\M\to\R$ defines a law-invariant functional $\w\rho\colon\CL^0\to\R$ by $\w\rho(X):=\rho(F^\P_X)$, where $F^\P_X:=\P(X\le \cdot)$ is the CDF of $X$ under $\P$. Conversely, each law-invariant functional $\rho\colon\CL^0\to\R$ induces a functional $\w\rho(G):=\rho(X_G)$, $G\in\mathcal M$, in a well-defined manner; here, $X_F$ is an arbitrary random variable with $F^{\P}_{X_F}=F$.

An important property of classical quantiles is their {\em ordinal covariance}. 
We recall its definition for a general functional $\rho \colon \CL^0 \to \R$ following, e.g., \cite{C09}.
Throughout, we set
$$
\Scal: = \{ g \colon \R \to \R \text { such that } g \text { is strictly increasing, continuous and surjective} \}.
$$
Algebraically, we view $(\Scal, \circ)$ as the group of orientation-preserving homeomorphisms of $\R$ with the operation of composition. 
\begin{definition}
Let $\rho \colon \CL^0 \to \R$. Then $\rho$ is \emph{ordinal covariant} if 
$$
\rho(g(X))=g(\rho(X)), \text { for all } g \in \Scal. 
$$ 
\end{definition}

This very strong property requires that the functional $\rho$ commutes with any strictly increasing and continuous bijection of $\R$ onto itself, including  nonlinear ones. 
It is satisfied by the usual quantiles, and indeed captures one of the facets of the ordinal nature of the usual quantiles. 
\cite{FLW23} show that the usual quantiles satisfy an even stronger ordinal covariance property, in which the group $\Scal$ is replaced by the monoid of increasing and left-continuous, but not necessarily surjective transformations. 
Both ordinal covariance properties lead to remarkable axiomatisations of classical quantiles in \cite{C09} and \cite{FLW23}, respectively. 
Somewhat relatedly, at the level of distributional transforms $T\colon\M\to\M$ as in the terminology of \cite{LSW2021}, ordinal covariance
corresponds to commutation with all shape transforms $F\mapsto F\circ g^{-1}$,
$g\in\Scal$. More recently, \cite{CLW25} characterise certain classes of distributional transforms through mutual commutation properties.

In this section we adopt an alternative point of view, regarding ordinal covariance not as an all-or-nothing property as above, but as a property that a functional can possess \emph{with a continuum of degrees}, which we aim to describe and compare.
We start with the following fundamental definition. 

\begin{definition}[Ordinal covariance group of a functional]

Let $\rho \colon \CL^0 \to \R$ be any functional. We define its \emph{ordinal covariance group} by 
\begin{equation}
\label{eq:G-rho}
\G(\rho):= \big\{ g \in \Scal \mid \rho(g(X))=g(\rho(X)), \text { for all } X \in \CL^0 \big \}.
\end{equation}
\end{definition}

\begin{lemma}\label{lem:group}
The set $ \G(\rho)$ defined in \eqref{eq:G-rho} is a subgroup of $\Scal$.  
\end{lemma}
\begin{proof}
Clearly, for any $\rho \colon \CL^0 \to \R$ it holds that $\mathrm{id}_\R \in \G(\rho)$, so  $\G(\rho) \neq \varnothing$. 
For $g,h\in\mathcal G(\rho)$ and $X\in\mathcal L^0$, we have 
\[ \rho\big((g\circ h)(X)\big)=g\big(\rho(h(X))\big)=(g\circ h)(\rho(X)),\]
which means that $g\circ h\in \mathcal G(\rho)$. 
Also, for $g\in\mathcal G(\rho)$, the inverse $g^{-1}$ is defined on all of $\R$ by surjectivity, and is strictly increasing and continuous. 
    Moreover, we have for all $X\in\mathcal L^0$ by definition that 
$
g^{-1}\big(\rho(X)\big)=g^{-1}\left(\rho\big((g\circ g^{-1})(X)\big)\right)=(g^{-1}\circ g)\left(\rho\big(g^{-1}(X)\big)\right)=\rho\big(g^{-1}(X)\big)$, hence $g^{-1}\in\G(\rho)$.
\end{proof}
The following example identifies the ordinal covariance groups of some well-known risk measures. For the sake of simplicity, we consider them on the space $\CL^\infty$ of all bounded random variables. All definitions relevant to this section directly carry over to this smaller space.

\begin{example}[Ordinal covariance groups of risk measures]\label{ex:OCGs}
Let $\E$ be the mean, $\mathrm{SD}$ the standard deviation, the entropic risk measure
$\rho_\gamma(X)=\tfrac1\gamma\log\E[e^{\gamma X}]$ ($\gamma\neq0$), and the worst-case risk measure given by $\rho^{\mathrm w}=Q_1^-$. Then
\[
\mathcal G(\E)=\mathrm{Aff}^+,\qquad
\mathcal G(\mathrm{SD})=\mathcal D,\qquad
\mathcal G(\rho_\gamma)=\mathcal T,\and\mathcal G(\rho^{\mathrm w})=\Scal,
\]
where
$
\mathcal T:=\{x\mapsto x+b\mid b\in\R\}$, $\mathcal D:=\{x\mapsto ax\mid a>0\}$,
$\mathrm{Aff}^+:=\{x\mapsto ax+b\mid a>0,\ b\in\R\}$
are respectively the translation, dilation, and orientation-preserving affine subgroups of $\Scal$. More generally, for coherent risk measures it holds that $\G(\rho) \supseteq \mathrm{Aff}^+$. 
The proof is postponed to Appendix~\ref{sec:example}. 
\end{example}

We will see below that these ordinal covariance groups are very different from those of $\La$-quantiles. The groups above act on
$\R$ with dense or transitive orbits and without any fixed
point, with the exception of $0$ in the case of standard deviation. In order to describe the ordinal covariance groups of $\La$-quantiles we give the following definition. 

\begin{definition}[Invariance group of $\La$]

Let $\La \colon \R \to [0,1]$. We define its \emph{invariance group} by
\begin{equation}
\G(\La):= \{g \in \Scal \mid \Lambda\circ g=\Lambda\}.  
\end{equation}
\end{definition}
As in the proof of Lemma \ref{lem:group}, it is straightforward to check that $\G(\La)$ is a subgroup of $\Scal$. 
The following example shows three prototypical cases. 

\begin{example}[Invariance groups of elementary $\La$]\label{ex:GLambda} We start by introducing the following notation. 
For $A\subseteq\R$ we set $\Scal_A:=\{g\in\Scal\mid g|_A=\mathrm{id}_A\}$, that is the pointwise
stabiliser of $A$ in $\Scal$. 
\begin{enumerate}[label=(\alph*)]
\item If $\La\equiv\lambda\in[0,1]$ is constant, then $\G(\La)=\Scal=\Scal_\varnothing$.
Indeed, if $\La\equiv\lambda$ then $\La\circ g=\lambda=\La$ for every
$g\in\Scal$, so $\G(\La)=\Scal$.
\item Let $-\infty=x_0<x_1<\dots<x_n<x_{n+1}=\infty$ and let
$\lambda_1,\dots,\lambda_{n+1}\in[0,1]$ satisfy $\lambda_i\neq\lambda_{i+1}$ for all $i$.
If $\La|_{(x_{i-1},x_i)}\equiv\lambda_i$ for $1\le i\le n+1$, i.e., $\La$ is locally
constant with jumps exactly at $x_1,\dots,x_n$, then
\[
\G(\La)=\Scal_{\{x_1,\dots,x_n\}}=\{g\in\Scal\mid g(x_i)=x_i,\ i=1,\dots,n\}.
\]
To verify this claim, suppose first that $g(x_i)=x_i$ for all $i$. 
In particular, $g$ maps each interval
$(x_{i-1},x_i)$ onto itself, on which $\La$ is constant; hence $\La\circ g=\La$ and
$\Scal_{\{x_1,\dots,x_n\}}\subseteq\G(\La)$. Conversely, let $g\in\Scal$ with
$g(x_i)\neq x_i$ for some $i$. By Lemma~\ref{lem:group}, replacing $g$ by $g^{-1}$ if
necessary, we may assume $g(x_i)>x_i$. Choosing $\delta>0$ small enough that
$g(x_i-\delta)>x_i$ as well,
\[
\La(x_i-\delta)=\lambda_i\neq\lambda_{i+1}=\La\big(g(x_i-\delta)\big),
\]
so $\La\circ g\neq\La$ and $g\notin\G(\La)$. Hence $\G(\La)=\Scal_{\{x_1,\dots,x_n\}}$.
\item If $\La$ is strictly monotone, then $\G(\La)=\{\mathrm{id}_\R\}=\Scal_\R$.
Indeed, a strictly monotone $\La$ is injective, so $\La\circ g=\La$
forces $g=\mathrm{id}_\R$; thus $\G(\La)=\{\mathrm{id}_\R\}$.
\end{enumerate}
\end{example}

The first main result of this section is that under minimal additional conditions on $\La$, the ordinal covariance group of a $\La$-quantile coincides with the invariance group of $\La$.  

\begin{theorem}\label{thm:ordinal}
Let $\rho = Q_\La^-$. If $\La$ is in standard representation, satisfies the attainment condition and leads to a finite left-quantile, then 
\[
\mathcal G(\rho)=\mathcal G(\La).
\]
\end{theorem}
\begin{proof}
    Let $g\in \G(\La)$. For all $F\in\mathcal M$, 
    \begin{align*} \rho(F\circ g^{-1})&=\inf\{x\in\R\mid F(g^{-1}(x))\ge \La(x)\}\\
    &=\inf\{g(x)\mid x\in\R,~F(x)\ge \La(g(x))\}\\
    &=\inf\{g(x)\mid x\in\R,~F(x)\ge \La(x)\}=g(\rho(F)),\end{align*}
showing that $g \in \G(\rho)$. 

For the converse implication, we claim that $N_0<\infty$ entails for all $g\in\mathcal G(\rho)$ that
$g((-\infty,N_0])=(-\infty,N_0]$ and $g((N_0,\infty))=(N_0,\infty)$. 
Indeed, let $x\le N_0$ and $B_{x,0}\in\M$ be defined by \eqref{eq:Bernoulli}, and compute
\[\rho(B_{g(x),0})=\rho(B_{x,0}\circ g^{-1})=g(\rho(B_{x,0}))=g(x).\]
As $\rho(B_{y,0})=N_0$ whenever $y>N_0$, we conclude $g((-\infty,N_0])\subseteq(-\infty,N_0]$.
Moreover, $\mathcal G(\rho)$ being a group implies that $g((N_0,\infty))\subseteq(N_0,\infty)$. The claimed identities follow from surjectivity of $g$.
    
Now assume towards a contradiction that we find $g\in\mathcal G(\rho)\setminus \mathcal G(\La)$, i.e., for some $x\in\R$, $\La(g(x))\neq \La(x)$, implying that $g(x)\neq x$ and $x\le N_0$. 
Since $\G(\rho)$ and $\G(\La)$ are subgroups of $\Scal$, also $g^{-1}\in\mathcal G(\rho)\setminus \mathcal G(\La)$, so we can assume without loss of generality that  $\La(g(x))>\La(x)$ and that $\La>0$ on $(-\infty,x)$. 
Let $F:=B_{x,1-\Lambda(x)}$, resulting in $\rho(F)=x$. 
Since $g \in \G(\rho)$,
    $\rho(F\circ g^{-1})=g(x)$.
However, $F\circ g^{-1}=\Lambda(x)\ind_{[g(x),g(x+1))}+\ind_{[g(x+1),\infty)}$, which entails 
  \[\La(x)=(F\circ g^{-1})(g(x))<\La(g(x)),\]
contradicting $\rho(F\circ g^{-1})=g(x)$ by Lemma~\ref{lem:attain}.
\end{proof}

Note that the inclusion $\G(\La) \subseteq \G(\rho)$ holds in general, while the additional assumptions of standard representation and attainment condition are only needed for the reverse inclusion.

For $\La$-quantiles, the assessment of their ordinal covariance properties is thus often equivalent to the identification of the invariance group of $\La$. 
Example \ref{ex:GLambda} showed that in simple examples these groups consist of the elements of $\Scal$ fixing every point of some set $A \subset \R$, i.e., are the \emph{pointwise stabilisers} of $A$. The next example shows that this is not always the case. 

\begin{example}[An invariance group that is not a pointwise stabiliser]\label{ex:periodic}
Let
\[
\La(x)=\tfrac12+\tfrac13\sin(2\pi x)\in\big[\tfrac16,\tfrac56\big]\subset(0,1),\quad x\in\R.
\]
Due to the periodic nature of $\La$, the invariance group  
$\G(\La)$ coincides with $\mathcal T_{\mathbb{Z}}:=\{x\mapsto x+b:\ b\in\mathbb{Z}\}$,
which is not the pointwise stabiliser of any subset of $\R$.
\end{example}

In order to make further progress, we assume that $\La$ has bounded variation and is right-continuous.   
In this case, $\La$ defines a bounded signed measure $\mu_{\La}$ on $(\R, \mathcal{B}(\R))$ by setting
\begin{equation*}
\mu_{\La}((-\infty,x]):=\La(x)-\La(-\infty);
\end{equation*}
see Lemma~\ref{lem:measure} in Appendix \ref{sec:app} for this and other useful properties of bounded variation functions. 
\begin{definition}[Measure-preserving transformation and $\G(\mu_\La)$]
Let $\La \colon \R \to [0,1]$ have bounded variation and be right-continuous, and let $\mu_\La$ be its associated signed measure as above. A measurable function $g \colon \R \to \R$ is a measure-preserving transformation of $(\R, \mathcal{B}(\R), \mu_\La)$ if, for each $A \in \mathcal{B}(\R)$, $\mu_\La\big( g^{-1}(A) \big) =\mu_\La (A).$
We denote by
$
\G(\mu_\La)$ the set of all $g\in\Scal$ that are measure-preserving transformations of $(\R, \mathcal{B}(\R), \mu_\La)$.
\end{definition}

The next theorem shows that under the BV and right-continuity assumptions, the invariance group of $\La$ coincides with the group of its measure-preserving transformations, that in turn is the pointwise stabiliser of the support.

\begin{theorem}
\label{thm:MPT}
    Suppose $\La \colon \R \to [0,1] $ has bounded variation and is right-continuous. For $g \in \Scal$,
    the following are equivalent: 
    \begin{enumerate}[label=\tn{(\alph*)}]
        \item $g\in\mathcal G(\La)$. 
        \item $g \in \G(\mu_\La)$. 
        \item $g|_{\tn{supp}(\mu_\La)}=\mathrm{id}_{\tn{supp}(\mu_\La)}$.
    \end{enumerate}
\end{theorem}
\begin{proof}
(c) implies (a): We want to show that for all $x \in \R$, it holds that $\La(g(x))=\La(x)$. If $x \in \tn{supp}(\mu_\La)$ this is immediate from (c). If instead $x \notin \tn{supp}(\mu_\La)$, let $I$ be the connected component of $\R\setminus\tn{supp}(\mu_\La)$ containing $x$; by Lemma~\ref{lem:measure}(c), $I$ is an open interval on which $\La$ is constant. Each finite endpoint of $I$ lies in $\tn{supp}(\mu_\La)$ and is thus a fixed point of $g$. 
As $g$ is increasing and fixes $\tn{supp}(\mu_\La)$ pointwise, it maps $I$ onto itself, so $g(x)\in I$ and $\La(g(x))=\La(x)$.

(a) implies (b): This follows with Lemma~\ref{lem:measure}(d).

(b) implies (c): Assume by contradiction that there exists $x \in \supp (\mu_{\La})$ such that $g(x) \neq x$. Without loss of generality, we may assume $g(x)>x$. 
Set $\delta:=\frac 1 2(g(x)-x)$ as well as 
\[
O:=g^{-1}\big((g(x)-\delta,\infty)\big)\cap (x-\delta,x+\delta),  
\]
an open neighbourhood of $x$. As $x\in\tn{supp}(\mu_\La)$, 
$\vert \mu_\La \vert (O)>0$. By Lemma~\ref{lem:measure}(e), $g$ also preserves the finite positive measure $|\mu_\La|$.
By construction and monotonicity of $g$, $g^n(y)\ge g(y)>g(x)-\delta\ge\sup O$ holds for all $y\in O$ and $n\ge 1$, so no point of $O$ ever returns to $O$. 
However, the latter is in direct contradiction to the Poincar\'e Recurrence Theorem~\cite[Theorem 20.3]{Ali} applied to $(\R,\mathcal B(\R),|\mu_\La|,g)$.
\end{proof}

Looking back at Example \ref{ex:GLambda}, we see that 
a constant $\La$ has $\mu_\La=0$ and empty support, giving the full group $\Scal$ and
recovering the classical quantiles; a strictly monotone $\La$ has
$\operatorname{supp}(\mu_\La)=\R$, giving the trivial group; the multilevel case lies in
between, with $\operatorname{supp}(\mu_\La)=\{x_1,\dots,x_n\}$ and a group as large as the
finitely many thresholds allow. In this sense we say that the size of $\G(\La)$ measures how far
$Q_\La^-$ departs from a classical quantile.
Notice that Example \ref{ex:periodic} shows that the conclusion is false without the assumption of bounded variation of $\La$. 

Summing up, 
the results of this section recast ordinal covariance from a binary property into a
group-valued invariant, giving a precise mathematical content to the  ``continuum of degrees'' anticipated above. By Theorem~\ref{thm:ordinal} the
ordinal covariance group $\G(Q_\La^-)$ coincides with the invariance group $\G(\La)$, and,
whenever $\La$ has bounded variation and is right-continuous, Theorem~\ref{thm:MPT} shows
that the latter is exactly the pointwise stabiliser $\Scal_{\tn{supp}(\mu_\La)}$ of the
support of the associated signed measure. The degree of ordinal covariance of a
$\La$-quantile is thus encoded in a single closed set, ranging from the full group $\Scal$
of the classical quantiles (empty support) to the trivial group of a strictly monotone
$\La$ (full support).
The measure-preserving viewpoint
of Theorem~\ref{thm:MPT} suggests the application of
tools from ergodic theory, a direction that we leave to future work.

\section{Conclusions and directions for further research}\label{sec:concl}

We have developed a systematic analysis of $\La$-quantiles under minimal assumptions
on the functional parameter $\La$, without the monotonicity assumption
standard in the literature. Two findings give \emph{a posteriori} an intrinsic
motivation to this generality. The first is that it is right lower semicontinuity of $\La$, not monotonicity, the property that drives most of our arguments. The second is that
closing the class of $\La$-quantiles under minimum aggregation, or under mixing with a fixed distribution at a fixed weight, forces the inclusion of
non-monotone $\La$. For the convenience of the reader, we recall our main results in Table~\ref{tab:summary}. 

\begin{table}[htpb]
\centering
\small
\renewcommand{\arraystretch}{1.3}
\begin{tabular}{@{}p{0.34\textwidth}p{0.40\textwidth}l@{}}
\toprule
\textbf{Assumption on $\La$} & \textbf{Property of the $\La$-quantiles} & \textbf{Result} \\
\midrule
right lower semicontinuous
  & infimum defining $Q_\La^{-}$ is a minimum
  & Lemma~\ref{lem:attain}\\
right lower semicontinuous
  & $\La$ is identified by the $\La$-quantiles of shifted Bernoulli distributions
  & Prop.~\ref{prop:geq}, \ref{prop:geq-2}\\
none
  & infimum and minimum aggregation rules of the classical quantiles carry over
  & Prop.~\ref{prop:lattice1}\\
$\La(x)\ge\liminf_{y\downarrow x}\La(y)$ for all $x$
  & $Q_\La^{+}$ is weakly upper semicontinuous
  & Thm.~\ref{thm:right-usc}(a)\\
equal to $1$ on a left half-line and strictly decreasing thereafter
  & $Q_\La^{+}$ is weakly continuous
  & Thm.~\ref{thm:right-usc}(b)\\
lower semicontinuous
  & $Q_\La^{-}$ is weakly lower semicontinuous
  & Thm.~\ref{thm:left-lsc}(a)\\
strictly decreasing until it vanishes, and zero afterwards
  & $Q_\La^{-}$ is weakly continuous
  & Thm.~\ref{thm:left-lsc}(b)\\
right lower semicontinuous 
  & $Q_\La^{-}$ has convex level sets
  & Prop.~\ref{prop:CxLS}\\
locally right monotone with no downward jump 
  & $Q_\La^{-}$ has convex level sets; the strict version of the same condition characterises $Q_\La^{+}$
  & Thm.~\ref{thm:CxLS}, Prop.~\ref{prop:CxLS+}\\
measurable
  & the $\La$-quantiles are elicitable on the distributions crossing $\La$ once
  & Thm.~\ref{thm:elicit}\\
bounded variation and right lower semicontinuous
  & mixture representation with  $\La$ increasing 
  & Thm.~\ref{thm:MW-1}\\
decreasing
  & mixture representation with constant $\Lambda$
  & Cor.~\ref{cor:MV}\\
right lower semicontinuous
  & ordinal covariance group of $Q_\La^{-}$ is the invariance group of $\La$
  & Thm.~\ref{thm:ordinal}\\
bounded variation and right-continuous
  & ordinal covariance group of $Q_\La^{-}$ is the pointwise stabiliser of the support of the signed measure associated with $\La$
  & Thm.~\ref{thm:MPT}\\
\bottomrule
\end{tabular}
\caption{Overview of our main results. Technical assumptions such as standard representation and finiteness
are omitted here and stated precisely in the body of the paper.}
\label{tab:summary}
\end{table}

Noticeably, we turn several previously known sufficient conditions into
equivalences---e.g., for weak semicontinuity and weak continuity of both quantiles,
and for the convex level set property---and identify the ordinal covariance group
of a $\La$-quantile with the invariance group of $\La$, which in the
bounded-variation case is the pointwise stabiliser of the support of the associated
signed measure, leading to interesting and nontrivial connections with ergodic theory.

The ordinal point of view naturally leads to the attempt to axiomatise subclasses of $\La$-quantiles through
their ordinal covariance groups.
Ideally, this should recover the classical quantiles when the group is the whole of $\Scal$ as in 
\cite{C09} and \cite{FLW23}.  
Ongoing work of the authors pursues this programme for
bilevel and multilevel quantiles. A first step in this direction is the
realisability question: which subgroups of $\Scal$ arise as the ordinal covariance
group of some law-invariant functional.  Theorem~\ref{thm:MPT} answers this for $\La$
of bounded variation, while Example~\ref{ex:periodic} shows that the picture beyond
bounded variation is richer.
Finally, the mixture representation of Section~\ref{sec:MR} makes the
numerical schemes developed in \cite{PW2026} for monotone $\La$ available also in the
general case, and suggests a general way to deal with non-monotone $\La$-quantiles in portfolio optimisation or risk management applications. 

\appendix

\section{Functions of bounded variation on \texorpdfstring{$\R$}{R}}
\label{sec:app}

A function $f\colon \R\to\R$ has bounded variation if 
\[\TV(f):=\sup\Big\{\sum_{i=1}^n|f(x_i)-f(x_{i-1})|\,\Big|\,n\in\N,\,x_0<x_1<...<x_n\Big\}<\infty.\]
A good reference for the properties of bounded variation functions defined on all of $\R$ is \cite{Fol99}. For completeness, we recall in the next lemma the properties that are needed. 

\begin{lemma}\label{lem:BV}
Suppose $\Lambda\colon\R\to\R$ has bounded variation. 
\begin{enumerate}[label=\tn{(\alph*)}]
    \item For each $x\in\R\cup\{-\infty\}$, the right limit $\La(x+)$ exists and is finite.
    \item For each $x\in\R\cup\{\infty\}$, the left limit $\La(x-)$ exists and is finite. 
    \item There exist two increasing and bounded functions $G^+,G^- \colon \R \to [0, \infty)$ satisfying \\
    $G^+(-\infty)=0$, $G^-(-\infty)=0$ such that  
\begin{equation}  
\label{eq:mon-diff}
\Lambda(x) =\La(-\infty) +G^+(x)-G^-(x).  
\end{equation}
    \item If $\Lambda$ is right-continuous, then $G^+$ and $G^-$ in the decomposition \eqref{eq:mon-diff} can be chosen to be right-continuous.

\item If $\La$ satisfies $\La(x)\le \La(x+)$ for all $x\in\R$, then in the decomposition \eqref{eq:mon-diff} $G^-$ can be chosen to be right-continuous.
\end{enumerate}
\end{lemma}

\begin{proof}
For (a)--(d) we refer to Section 3.5 of \cite{Fol99}. 
For (e), write $\La=\La(-\infty)+G^+-G^-$ and let $x$ be such that $G^-$ is not continuous at $x$, i.e., $G^-(x+)>G^-(x)$. 
    As $0\le \La(x+)-\La(x)=G^+(x+)-G^+(x)-|G^-(x+)-G^-(x)|$, 
    $G^+$ must also be discontinuous at $x$ and its jump must be at least as large as the one of $G^-$.
    Moreover, 
    \[G^+(y)\ge G^+(x+)\ge G^+(x)+G^-(x+)-G^-(x),\quad y>x.\]
    Hence, modify $G^{\pm}$ to be 
    \[\widetilde G^-(y):=\begin{cases}G^-(y)&\text{if }y\neq x,\\
    G^-(x+)&\text{if }y=x,\end{cases}\qquad \widetilde G^+(y):=\begin{cases}G^+(y)&\text{if }y\neq x,\\
    G^+(x)+G^-(x+)-G^-(x)&\text{if }y=x.\end{cases}
    \]
    Do so with all countably many discontinuities of $G^-$. 
\end{proof}

\begin{lemma}\label{lem:measure}
Let $\La\colon\R\to\R$ be right-continuous and of bounded variation.
\begin{enumerate}[label=\tn{(\alph*)}]
\item There is a unique finite signed measure $\mu_\La$ on $(\R,\mathcal B(\R))$ satisfying 
\[
\mu_\La\big((-\infty,x]\big)=\La(x)-\La(-\infty),\qquad x\in\R.
\]
In particular, among functions sharing the same
value $\La(-\infty)$, the assignment $\La\mapsto\mu_\La$ is injective.
\item The total variation measure satisfies $|\mu_\La|(I)=\TV(\La;I)$ for every open
interval $I\subseteq\R$; hence $|\mu_\La|(\R)=\TV(\La)$, and $\mu_\La=0$ if and only if
$\La$ is constant.
\item The support $\tn{supp}(\mu_\La)$ is closed, and $x\notin\tn{supp}(\mu_\La)$ if and
only if $\La$ is constant on some neighbourhood of $x$. Consequently
$\R\setminus\tn{supp}(\mu_\La)$ is the union of the at most countably many maximal open
intervals on which $\La$ is constant, and $\La$ is constant on each connected component of
$\R\setminus\tn{supp}(\mu_\La)$.
\item For every $g\in\Scal$, the composition $\La\circ g$ is right-continuous, of
bounded variation, and satisfies
\[
\TV(\La\circ g)=\TV(\La)\and\mu_{\La\circ g}(A)=\mu_\La\big(g(A)\big),~A\in\mathcal B(\R).
\]
In particular $\tn{supp}(\mu_{\La\circ g})=g^{-1}\big(\tn{supp}(\mu_\La)\big)$ and
$\La\circ g=\La$ if and only if $\mu_\La(g^{-1}(A))=\mu_\La(A)$
for all $A\in\mathcal B(\R)$.
\item Any $g\in\mathcal S$ preserves $\mu_\La$ only if it also preserves $|\mu_\La|$. 
\end{enumerate}
\end{lemma}

\begin{proof}
(a) A right-continuous function of bounded variation is a difference of two bounded
increasing right-continuous functions by Lemma~\ref{lem:BV}\tn{(d)}, and each such
function is the cumulative function of a unique finite Borel measure by the
Lebesgue--Stieltjes correspondence \cite[\S3.5]{Fol99}; taking their difference yields
$\mu_\La$. 
Injectivity follows from $\La(x)=\La(-\infty)+\mu_\La((-\infty,x])$, $x\in\R$. 

(b) The identity $|\mu_\La|(I)=\TV(\La;I)$ is the Lebesgue--Stieltjes description of
the total variation measure \cite[\S3.5]{Fol99}; the two consequences follow by taking
$I=\R$ and noting that $\TV(\La)=0$ characterises constant functions.

(c) By definition $x$ lies in $\R\setminus\tn{supp}(\mu_\La)$ if and only if some open neighbourhood
$U$ of $x$ satisfies $|\mu_\La|(U)=0$. 
By (b), the latter is equivalent to $\La$ being constant on $U$. Thus
$\R\setminus\tn{supp}(\mu_\La)$ is the largest open set on which $\La$ is locally constant;
its connected components are open intervals, at most countably many since pairwise
disjoint, and on each of them $\La$ is
constant.

(d) As $g\in\Scal$ is an increasing homeomorphism, $g(-\infty)=-\infty$, so
$\La\circ g$ is right-continuous with $(\La\circ g)(-\infty)=\La(-\infty)$. The
increasing reparametrisation leaves the total variation unchanged, i.e., $\TV(\La\circ g)=\TV(\La)<\infty$. For every $x\in\R$,
\[
\mu_{\La\circ g}\big((-\infty,x]\big)=\La(g(x))-\La(-\infty)
=\mu_\La\big((-\infty,g(x)]\big)=\mu_\La\big(g((-\infty,x])\big).
\]
Both $\mu_{\La\circ g}$ and $A\mapsto\mu_\La(g(A))$ are finite signed measures, and a
finite signed measure is determined by its values on the half-lines $(-\infty,x]$; hence
$\mu_\La=\mu_{\La\circ g}$.
The support identity follows since $g$ is a
homeomorphism, and the final equivalence combines this with the injectivity in \tn{(a)}.

(e) A $g\in\Scal$ preserves $\mu_\La$ if and only if $\La\circ g=\La$ by (b). Hence, for $x<y$,  
\[|\mu_\La|\big(g^{-1}\big((x,y)\big)\big)=\TV\big(\La\circ g;\big(g^{-1}(x),g^{-1}(y)\big)\big)=\TV\big(\La;(x,y)\big)=|\mu_\La|(x,y).\]
\end{proof}

\section{On local right monotonicity}
\label{sec:ass}
The following lemma provides a deeper analysis of the property of local right monotonicity from Definition~\ref{def:LRM}. 

\begin{lemma}\label{lem:structure}
Let $\La\colon\R\to[0,1]$. Then $\La$ is locally right monotone if and only if
there exists a family $\mathcal J$ of pairwise disjoint open
        intervals such that:
        \begin{enumerate}[label=\tn{(\roman*)}]
            \item $\La$ is monotone on each interval in $\mathcal J$.
            \item The complement
            \[C:=\R\setminus\bigcup_{I\in\mathcal J}I\]
            is right-isolated in the sense that for all $x \in C$ there exist $y > x$ such that $(x,y) \cap C = \varnothing$. 
        \end{enumerate}
\end{lemma}
\begin{proof}
    To show sufficiency of properties (i) and (ii),  pick $x \in \R$. There are two cases. If $x\in\R\setminus C$, it belongs to an interval
    $I\in\mathcal J$ on which $\La$ is monotone; hence $\La$ is monotone on an
    open interval with left endpoint $x$.
    Else, if $x\in C$, property (ii) delivers a $y>x$ such that
    $(x,y)\cap C=\varnothing$, i.e., $(x,y) \subseteq \bigcup_{I\in\mathcal J}I$. Since the open intervals in $\mathcal J$ are disjoint and $(x,y)$ is connected, it is contained in exactly one element of $\mathcal J$, and local right monotonicity follows.
    
    To prove necessity, let $M$ be the set of all open intervals on which $\La$ is
    monotone. This set is nonempty by definition of local right monotonicity, partially ordered by set inclusion, and chains in $M$
    have upper bounds: the union of a chain of nested monotone intervals is again
    a monotone interval, its monotonicity direction being that of any of its
    non-constant members. 
    By Zorn's lemma, $M$ contains maximal elements.
    Fixing a monotone interval $(x,y)$ and applying Zorn's lemma to
    $M':=\{I\in M\mid (x,y)\subseteq I\}$, this set has maximal elements as well,
    and a maximal element of $M'$ is maximal in $M$; hence every interval in $M$
    lies in a maximal one.
    Distinct maximal monotone intervals cannot contain one another, and if two of them have a nonempty intersection and the same monotonicity behaviour, this would contradict their maximality. 
    Hence, two different maximal elements can only have a nonempty intersection if $\La$ is constant on the said intersection. 
    
In order to arrive at the disjoint family $\mathcal J$ of the thesis, we have to cut these intersections in a suitable way. Note first that each maximal $I=(a,b)\in M$ can intersect at most other two such
intervals, one on the left and one on the right. Indeed, a maximal interval
intersecting $I$ cannot be contained in it or contain it, hence it extends beyond
exactly one of its endpoints. Assume that $I_1=(a_1,b_1)$ and $I_2=(a_2,b_2)$ both meet
$I$ from the left, i.e.\ $a_1,a_2<a$. Then the direction of $\La$ on
$I_1$ is opposite to that on $I$, and similarly the direction of $\La$ on $I_2$ is opposite to that on $I$; hence $\La$ has the same monotonicity behaviour on $I_1$ and
$I_2$, and maximality implies $I_1=I_2$. The same argument applies on
the right.
    
    Cutting each overlap of maximal intervals at
    one interior point turns the family of maximal monotone intervals into a pairwise
    disjoint family $\mathcal J$ of monotone open intervals whose union
    $\bigcup_{I\in\mathcal J}I$ is $\bigcup M$ deprived of the at most countable
    set of cut points. Set $C:=\R\setminus\bigcup_{I\in\mathcal J}I$ and pick $x\in C$. If
    $x\notin\bigcup M$, then by local right monotonicity there is $y>x$ with $\La$
    monotone on $(x,y)$, and the maximal monotone interval containing $(x,y)$ has
    left endpoint $x$ (otherwise $x\in\bigcup M$), that gives right-isolation; if instead $x\in\bigcup M$, then
    $x$ is one of the cut points, chosen in the interior of a constant overlap, so 
    $(x,x+s)\cap C=\varnothing$ for some $s>0$.
\end{proof}

\begin{remark}
\begin{enumerate}[label=(\alph*)]

    \item The set $C$ is closed, and the right-isolation property forces it to be countable: sending each $x\in C$ to a rational lying in a right
    neighbourhood $(x,x+s)$ disjoint from $C$ defines an injection into $\Q$. Being
    closed and countable, $C$ is nowhere dense and has zero Lebesgue measure; equivalently,
    $\bigcup_{I\in\mathcal J}I=\R\setminus C$ is open, dense and of full
    measure, so $\La$ is monotone on a neighbourhood of almost every point.

    \item No point of $C$ is
    an accumulation point of $C$ from the right, but $C$ may cluster from the
    left and need not be discrete. For instance $C=\{0\}\cup\{-\frac 1 n\mid n\in\N\}$ is
    possible.

    \item The set $C$ marks the separations between the maximal intervals on which
    $\La$ is monotone. Hence the strict local extrema of $\La$ form a subset of $C$.

    \item Locally right monotone functions $\La$ also enjoy a
    property reminiscent of the {\em pseudomonotonicity} of
    \cite[Proposition~1.15]{App}: for every interval $J\subseteq[0,1]$, the preimage
    $\La^{-1}(J)$ is a \emph{countable} (instead of finite) union of intervals.
\end{enumerate}
\end{remark}

\section{Ordinal covariance groups of risk measures}\label{sec:example}

In this section we derive more formally the ordinal covariance groups of Example \ref{ex:OCGs}. 

\subsubsection*{Mean} If $g(x)=ax+b$ then $\E[g(X)]=a\E[X]+b=g(\E[X])$, so
$\mathrm{Aff}^+\subseteq\mathcal G(\E)$. Conversely, let $g\in\mathcal G(\E)$ and apply
$\E[g(X)]=g(\E[X])$ to the two-point law $X=x_1$ with probability $t$ and $X=x_2$ with
probability $1-t$, obtaining
$
t\,g(x_1)+(1-t)\,g(x_2)=g\big(t x_1+(1-t)x_2\big)
$, for all $x_1,x_2\in\R$ and $t\in[0,1].$
This is Jensen's functional equation; its monotone solutions are affine, and $g\in\Scal$
increasing forces the slope to be positive. Hence $\mathcal G(\E)=\mathrm{Aff}^+$.

\subsubsection*{Standard deviation} 
Since $\mathrm{SD}(aX)=a\,\mathrm{SD}(X)$ for $a>0$, dilations
$g(x)=ax$ satisfy $\mathrm{SD}(g(X))=a\,\mathrm{SD}(X)=g(\mathrm{SD}(X))$, so
$\mathcal D\subseteq\mathcal G(\mathrm{SD})$. Conversely, let
$g\in\mathcal G(\mathrm{SD})$.
Let $x\in\R$ and infer from ordinal covariance that 
\[\mathrm{SD}(x)=0=\mathrm{SD}(g(x))=g(\mathrm{SD}(x))=g(0).\]
Now let $y>0$ and suppose that $\P(X=x)=\P(X=x+2y)=\frac 1 2$. 
Then $\mathrm{SD}(X)=y$ and
$\mathrm{SD}(g(X))=\tfrac12\big(g(x+2y)-g(x)\big)$. Ordinal covariance
gives
\[g(x+2y)-g(x)=2g(y).\]
Setting $x=0$ and $x=-y$
yields $g(2y)=2g(y)$ and $g(-y)=-g(y)$, 
i.e., we can rewrite the equation above as 
\[g(x+y)-g(x)=g(y),\quad\text{for all }x\in\R,~y\ge 0.\]
By continuity and monotonicity, $g(x)=ax+g(0)=ax$ for some $a>0$, i.e., $\mathcal G(\mathrm{SD})=\mathcal D$. 

\subsubsection*{Entropic} Translation invariance $\rho_\gamma(X+b)=\rho_\gamma(X)+b$ gives
$\mathcal T\subseteq\mathcal G(\rho_\gamma)$. For the reverse inclusion we use the
certainty-equivalent form $\rho_\gamma=u^{-1}\circ\E\circ u$ with $u(x)=e^{\gamma x}$.
For $g\in\Scal$ set $h:=u\circ g\circ u^{-1}$ on $u(\R)=(0,\infty)$ and let $Y:=u(X)$
range over all random variables whose laws have finite support in $(0,\infty)$. Then $u(g(X))=h(Y)$, so
$
\rho_\gamma(g(X))=u^{-1}\big(\E[h(Y)]\big)
$
and
$g(\rho_\gamma(X))=u^{-1}\big(h(\E[Y])\big),
$
and since $u^{-1}$ is injective, $g\in\mathcal G(\rho_\gamma)$ iff
$\E[h(Y)]=h(\E[Y])$ for all such $Y$. By the argument for the mean, this is tantamount to $h$
being affine. 
As an increasing self-bijection of $(0,\infty)$, it must be of the form
$h(y)=\alpha y$ with $\alpha>0$. Therefore
$g(x)=u^{-1}(\alpha\,u(x))=x+\tfrac1\gamma\log\alpha$, i.e.\ $g\in\mathcal T$. Hence
$\mathcal G(\rho_\gamma)=\mathcal T$.

\subsubsection*{Worst case} As $\rho^{\mathrm w}=Q_1^-$ is the $\La$-quantile associated 
with the constant function $\La\equiv 1$, Example~\ref{ex:GLambda}(a) yields 
$\G(\La)=\Scal$. The inclusion $\G(\La)\subseteq\G(Q_\La^-)$, established in the first 
part of the proof of Theorem~\ref{thm:ordinal} without further assumptions on $\La$, 
gives $\G(\rho^{\mathrm w})=\Scal$. Alternatively, one checks directly that 
$\esssup g(X)=g(\esssup X)$ for all $g\in\Scal$: writing $s:=\esssup X$, the inequality 
``$\le$'' holds because $\P(X>s)=0$ forces $\P(g(X)>g(s))=0$, while ``$\ge$'' holds 
because $\P(X>s-\eps)>0$ gives $\esssup g(X)\ge g(s-\eps)$ for every $\eps>0$, and $g$ 
is continuous. 

\section*{Acknowledgements}
We are grateful to the participants of the 27th Workshop on
Quantitative Finance (Bergamo, March 2026), of the 5th Italian Meeting on
Probability and Mathematical Statistics (Palermo, June 2026), of the 13th World
Congress of the Bachelier Finance Society (Bologna, July 2026), and of the
FRIT2026 conference (Stresa, September 2026) for helpful comments and
discussions.

\end{document}